\documentclass[a4paper, reqno, usenames, dvipsnames]{amsart}

\usepackage{amsmath,amssymb,amsthm}
\usepackage{wrapfig}
\usepackage{fullpage}

\usepackage{microtype}
\usepackage[colorlinks=true,urlcolor=RoyalBlue,
linkcolor=RoyalBlue,citecolor=RoyalBlue]{hyperref}

\usepackage{mathtools}
\usepackage[normalem]{ulem}
\usepackage{graphicx,caption}
\usepackage{xypic}
\entrymodifiers={+!!<0pt,\fontdimen22\textfont2>}
\usepackage[all]{xy}
\usepackage[usenames,dvipsnames,svgnames,table]{xcolor}
\usepackage{tikz}
\usetikzlibrary{arrows,decorations.markings}
\usepackage{marvosym} 

\newtheoremstyle{myremark} 
    {7pt}                    
    {7pt}                    
    {}  	                 
    {}                           
    {\bf}       	         
    {.}                          
    {.5em}                       
    {}  

\makeatletter
\def\paragraph{\@startsection{paragraph}{4}%
  \z@\z@{-\fontdimen2\font}%
  {\normalfont\bfseries}}
\makeatother

\theoremstyle{plain}
\newtheorem{lemma}{Lemma}[section]
\newtheorem{claim}[lemma]{Claim}
\newtheorem{corollary}[lemma]{Corollary}
\newtheorem{proposition}[lemma]{Proposition}
\newtheorem{theorem}[lemma]{Theorem}
\newtheorem*{theorem-main}{Theorem~\ref{thm:main}}
\newtheorem*{theorem-secondary}{Theorem~\ref{thm:secondary}}
\theoremstyle{definition}
\newtheorem{conjecture}[lemma]{Conjecture}
\newtheorem{definition}[lemma]{Definition}
\newtheorem{example}[lemma]{Example}
\newtheorem{question}[lemma]{Question}
\theoremstyle{myremark}
\newtheorem{remark}[lemma]{Remark}

\newcommand{\N}{\mathbb{N}}
\newcommand{\R}{\mathbb{R}}

\newcommand{\F}{\mathbb{F}}

\newcommand{\cl}{\mathrm{Cl}}

\newcommand{\diredge}{\rightarrow}
\newcommand{\vrless}[2]{\mathrm{VR}_<(#1;#2)}
\newcommand{\vrleq}[2]{\mathrm{VR}_\leq(#1;#2)}
\newcommand{\vr}[2]{\mathrm{VR}(#1;#2)}
\newcommand{\vrcless}[2]{\mathbf{VR}_<(#1;#2)}
\newcommand{\vrcleq}[2]{\mathbf{VR}_\leq(#1;#2)}
\newcommand{\vrc}[2]{\mathbf{VR}(#1;#2)}
\newcommand{\wf}{\mathrm{wf}}
\newcommand{\cnk}[2]{C_{#1}^{#2}}

\newcommand{\olen}{\ell}

\newcommand{\numorb}{P}
\newcommand{\hit}{\mathcal{H}}
\newcommand{\fin}{\mathrm{Fin}}
\newcommand{\tfin}
{\widetilde{\mathrm{Fin}}}
\newcommand{\tW}{\widetilde{W}}
\newcommand{\tG}{\tilde{G}}
\newcommand{\tf}{\tilde{f}}
\newcommand{\tr}{\tilde{r}}
\newcommand{\tu}{\tilde{u}}
\newcommand{\tv}{\tilde{v}}
\newcommand{\dv}{d_V} 
\newcommand{\ds}{\vec{d}_{S^1}} 
\newcommand{\dpn}{\vec{d}_{P_n}} 
\newcommand{\gav}{\gamma_V} 
\newcommand{\gas}{\gamma} 
\newcommand{\gv}{g} 
\newcommand{\gpn}{g} 
\newcommand{\fs}{f}
\renewcommand{\sl}{s_{2l+1}} 
\newcommand{\qnl}{q_{n,l}}

\renewcommand{\r}{\rightarrow}
\newcommand{\norm}[1]{\|#1\|} 

\DeclareMathOperator*{\argmax}{arg\,max}

\DeclareMathOperator{\lcm}{lcm}
\DeclareMathOperator{\PH}{PH}
\DeclareMathOperator{\VR}{VR}

\newcommand{\image}{\mathrm{im}}
\newcommand{\rank}{\mathrm{rank}}

\newcommand{\diam}{\mathrm{diam}}
\newcommand{\proj}{\mathrm{proj}}

\begin{document}
\title{Vietoris--Rips complexes of regular polygons}
\author{Henry Adams}
\address{Department of Mathematics, Colorado State University, Fort Collins, CO 80523, United States}
\email{adams@math.colostate.edu}
\author{Samir Chowdhury}
\address{Department of Mathematics, Ohio State University, Columbus, OH 43210, United States}
\email{chowdhury.57@osu.edu}
\author{Adam Quinn Jaffe}
\address{Department of Mathematics, Stanford University, Stanford, CA 94305, United States}
\email{aqjaffe@stanford.edu}
\author{Bonginkosi Sibanda}
\address{Department of Mathematics, Brown University, Providence, RI 02912, United States}
\email{bonginkosi\_sibanda@brown.edu}
\thanks{This material is based upon work supported by the National Science Foundation under Grant No.\ DMS-1439786 while the authors were in residence at the Institute for Computational and Experimental Research in Mathematics in Providence, RI, during the Summer\MVAt ICERM 2017 program. While in residence at Summer\MVAt ICERM 2017, Bonginkosi Sibanda was also supported by The Karen T.\ Romer Undergraduate Teaching and Research Awards.}
\keywords{persistent homology, regular polygons, Vietoris--Rips complex, flag complex, homotopy type.}

\begin{abstract}
Persistent homology has emerged as a novel tool for data analysis in the past two decades. However, there are still very few shapes or even manifolds whose persistent homology barcodes (say of the Vietoris--Rips complex) are fully known. Towards this direction, let $P_n$ be the boundary of a regular polygon in the plane with $n$ sides; we describe the homotopy types of Vietoris--Rips complexes of $P_n$. Indeed, when $n=(k+1)!!$ is an odd double factorial, we provide a complete characterization of the homotopy types and persistent homology of the Vietoris--Rips complexes of $P_n$ up to a scale parameter $r_n$, where $r_n$ approaches the diameter of $P_n$ as $n\to\infty$. Surprisingly, these homotopy types include spheres of all dimensions. Roughly speaking, the number of higher-dimensional spheres appearing is linked to the number of equilateral (but not necessarily equiangular) stars that can be inscribed into $P_n$. As our main tool we use the recently-developed theory of cyclic graphs and winding fractions.
Furthermore, we show that the Vietoris--Rips complex of an arbitrarily dense subset of $P_n$ need not be homotopy equivalent to the Vietoris--Rips complex of $P_n$ itself, and indeed, these two complexes can have different homology groups in arbitrarily high dimensions. As an application of our results, we provide a lower bound on the Gromov--Hausdorff distance between $P_n$ and the circle.
\end{abstract}
\maketitle


\section{Introduction}

Let $S^1=\{(x,y)\in\R^2~|~x^2+y^2=1\}$ be the unit circle in the plane. Given an integer $n\ge 3$, let $P_n\subseteq \R^2$ be the boundary of the regular polygon inscribed in $S^1$ with $n$ vertices. In other words, $P_n$ is a piecewise-linear closed curve in the complex plane with vertices the $n$-th roots of unity. We equip $P_n$ with the Euclidean metric. What can be said about the homotopy types and the persistent homology of the Vietoris--Rips simplicial complexes $\vrc{P_n}{r}$?

One motivation for such questions is the application of topology to data analysis~\cite{EdelsbrunnerHarer,Carlsson2009}.
Suppose one is given a finite data set $X$ sampled from some unknown underlying infinite metric space $M$, and would like to use $X$ to recover information about $M$. For example, one can use \emph{persistent homology} to attempt to recover the homology groups of $M$~\cite{EdelsbrunnerHarer,ChazalOudot2008}. Though data set $X$ is typically finite, as the density of $X$ increases, the persistent homology of $\vrc{X}{r}$ converges to the persistent homology of $\vrc{M}{r}$~\cite{ChazalDeSilvaOudot2013}. Hence Vietoris--Rips complexes of infinite metric spaces are important, as they are the limiting objects of Vietoris--Rips complexes of finite data sets. Nevertheless, extremely little is known about the persistent homology of Vietoris--Rips complexes of basic shapes. In this paper, we develop the tools necessary to describe the homotopy types and persistent homology of Vietoris--Rips complexes of regular polygons.

Another motivation for such questions is the application of topology to machine learning.
One reason for incorporating more mathematics in machine learning is to try to improve not only the predictive power of machine learning algorithms, but also their interpretability. 
There are by now a wide variety of ways to turn the output from persistent homology into feature vectors for a machine learning task; see for example~\cite{adams2017persistence,bubenik2015statistical,carriere2018metric,chazal2018density,chevyrev2018persistence,di2015comparing,hofer2017deep,kalivsnik2018tropical,reininghaus2015stable,vskraba2018persistent}.
Nevertheless, the interpretation of persistent homology \emph{barcodes} (or equivalently \emph{diagrams}, see~\cite{chazal2016structure}) at larger scale parameters is still under development, even though we expect to be able to extract more geometric information about a space by considering its persistence not only at small scales but also at large scales~\cite[Section~12]{virk20171}. As a starting example, in this paper we detail the precise geometric information that higher-dimensional persisistent homology measures in a regular polygon.

Our main tool will be the structure of \emph{cyclic graphs} and their clique complexes. Roughly speaking, a cyclic graph is a directed graph in which the vertex set is equipped with a cyclic order, such that whenever there is a directed edge $u\to v$, then there is also a directed edge $u\to w\to v$ for all vertices $w$ cyclically between $u$ and $v$. The clique complex (of the underlying undirected graph) is the simplicial complex on the same vertex set, with faces given by the cliques (complete subgraphs) of the undirected graph. If $G$ is a cyclic graph then it is known that $\cl(G)$ is either an odd sphere or a wedge sum of even spheres of the same dimension~\cite{AAFPP-J}. Quantitative control over the dimension of the spheres is given in~\cite{AA-VRS1} using the \emph{winding fraction} of $G$.

For each $n\ge 4$, there exists a scale parameter $r_n>0$ such that the 1-skeleton of $\vrc{P_n}{r}$ is a cyclic graph for all $r< r_n$. This allows us to partially describe the homotopy types of $\vrc{P_n}{r}$ when $r< r_n$. The main result of our paper is the following theorem, where explicit formulas for the scale parameters
\[ 0=t_{n,0}\le s_{n,1}\le t_{n,1}\le s_{n,2}\le t_{n,2}\le\ldots\le s_{n,l}\le t_{n,l} \]
are given in Remark~\ref{rem:expl-sl}.

\begin{theorem-main}
Suppose $n=q(2l+1)$ is a multiple of $2l+1$. Let $r<r_n$. Then
\begin{align*}
\vrcless{P_n}{r}&\simeq\begin{cases}
\bigvee^{q-1}S^{2l}&\mbox{ when }s_{n,l}<r\le t_{n,l}\\
S^{2l+1}&\mbox{ when }t_{n,l}<r\le s_{n,l+1}
\end{cases}\\
\vrcleq{P_n}{r}&\simeq\begin{cases}
\bigvee^{q-1}S^{2l}&\mbox{ when }r=s_{n,l}\\ 
\bigvee^{3q-1}S^{2l}&\mbox{ when }s_{n,l}<r< t_{n,l}\\
\bigvee^{2q-1}S^{2l}&\mbox{ when }r=t_{n,l}\\
S^{2l+1}&\mbox{ when }t_{n,l}<r< s_{n,l+1}.
\end{cases}
\end{align*}
Furthermore, when $r<\tr<r_n$,
\begin{itemize}
\item For $s_{n,l}<r<\tr\le t_{n,l}$ or $t_{n,l}<r<\tr\le s_{n,l+1}$, inclusion $\vrcless{P_n}{r}\hookrightarrow\vrcless{P_n}{\tr}$ is a homotopy equivalence.
\item For $t_{n,l}<r<\tr< s_{n,l+1}$, inclusion $\vrcleq{P_n}{r}\hookrightarrow\vrcleq{P_n}{\tr}$ is a homotopy equivalence.
\item For $s_{n,l}\le r<\tr\le t_{n,l}$, inclusion $\vrcleq{P_n}{r}\hookrightarrow\vrcleq{P_n}{\tr}$ induces a rank $q-1$ map on $2l$-dimensional homology $H_{2l}(-;\F)$ for any field $\F$.
\end{itemize}
\end{theorem-main}

\begin{figure}[h]
\centering
\includegraphics[width=0.7\textwidth]{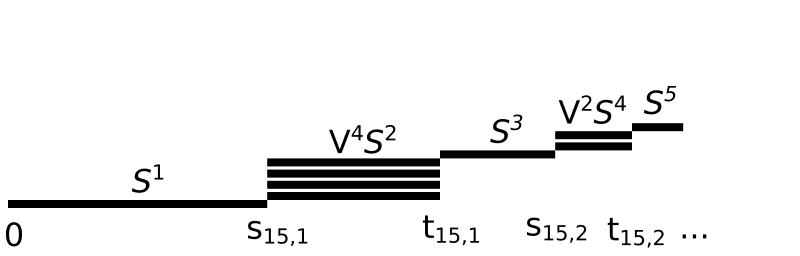}
\includegraphics[width=0.7\textwidth]{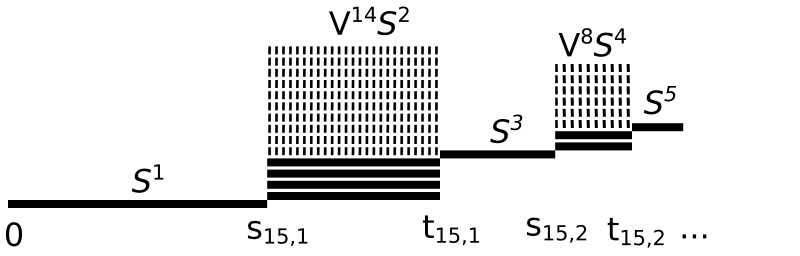}
\caption{(Top) The persistent homology of $\vrcless{P_{15}}{r}$ for $0\le r\le r_n$. The value of $r$ is on the horizontal axis, and increasing from left to right. (Bottom) The persistent homology of $\vrcleq{P_{15}}{r}$, in which the dotted lines correspond to \emph{ephemeral} summands.}
\label{fig:P15PersistentHomotopyRips}
\end{figure}

We remark that the $2l$-dimensional persistent homology module of $\vrcleq{P_n}{r}$ has an \emph{ephemeral summand} (\cite{ChazalCrawley-BoeveyDeSilva}) of rank $2q$ over the interval of scale parameters $s_{n,l}<r< t_{n,l}$. As a consequence of the main theorem, we can completely describe the homotopy types and the persistent homology of $\vrc{P_n}{r}$ when $n=(2k+1)!!=(2k+1)\cdot(2k-1)\cdot\ldots\cdot 5\cdot 3\cdot 1$ and the scale parameter is less than $r_n$; see Corollary~\ref{cor:main2}.

\begin{remark}
\label{rem:results}
In addition to Theorem~\ref{thm:main}, we also prove an analogous result for the case $l = 1$, cf.\ Theorem~\ref{thm:main-1}. The proof of the case of general $l$ is presented in Theorem~\ref{thm:main-general}, but this general case relies on our Conjecture~\ref{conj:monotonic}. Even though we have experimental results suggesting the validity of this conjecture, a full proof remains open.
\end{remark}

Let $S^1$ be the unit circle equipped with the Euclidean metric. As $n\to\infty$, the metric polygons $P_n$ converge in the Gromov-Hausdorff distance to the circle $S^1$. It follows from stability~\cite{ChazalDeSilvaOudot2013} that the persistent homology of $\vrc{P_n}{r}$ converges to that of $\vrc{S^1}{r}$. The persistent homology of $\vrc{S^1}{r}$ is known: as $r$ increases, $\vrc{S^1}{r}$ obtains the homotopy types of $S^1$, $S^3$, $S^5$, $S^7$, \ldots, until finally it is contractible~\cite{AA-VRS1}. In this paper we study the topological features present in the persistent homology of $\vrc{P_n}{r}$ before achieving convergence. One interesting observation is that even though $\vrc{S^1}{r}$ contains an $i$-dimensional persistent homology interval of positive length if and only if $i$ is odd, the polygon shapes $P_n$ provide examples of arbitrarily close metric spaces homeomorphic to $S^1$ which can contain nontrivial persistent $i$-dimensional homology even for $i$ even and arbitrarily large.

The stability of persistent homology guarantees that the Gromov-Hausdorff distance between two compact metric spaces is bounded below by half the \emph{bottleneck distance} between the $i$-dimensional persistent homology barcodes of the Vietoris-Rips complexes of the two spaces \cite{ChazalDeSilvaOudot2013}. The Gromov-Hausdorff distance is in general NP-hard to compute~\cite{memoli2007use}, whereas both the persistent homology and bottleneck distance computations can be carried out in polynomial time~\cite{edelsbrunner2000topological,zomorodian2005computing,munkres1957algorithms}. Therefore, knowledge of the persistent homology of $\vrc{P_n}{r}$ enables us to provide a lower bound on the Gromov-Hausdorff distance between $S^1$ and $P_n$ (Section~\ref{sec:GH}).

As a secondary result, we study finite subsets of the regular polygons. This work is related to Latschev's Theorem~\cite[Theorem~1.1]{Latschev2001}, which states that if $M$ is a Riemannian manifold and scale $r$ is sufficiently small, then for any sufficiently dense $X\subseteq M$ we have a homotopy equivalence $\vrcless{X}{r}\simeq M$. By Hausmann's Theorem~\cite{Hausmann1995} we also have $M \simeq \vrcless{M}{r}$, giving $\vrcless{X}{r}\simeq\vrcless{M}{r}$ for $r$ sufficiently small and $X$ sufficiently dense. In Question~\ref{ques:Latschev-higher} we ask: for $M$ a Riemannian manifold, is it also true at larger scale parameters $r$ that $\vrcless{X}{r}\simeq\vrcless{M}{r}$ for $X$ sufficiently dense depending on $r$? This is known to be true in the case when $M=S^1$ is the circle~\cite{AA-VRS1}, but to our knowledge this question is unknown for a general Riemannian manifold. Our Theorem~\ref{thm:secondary}(ii) shows that Question~\ref{ques:Latschev-higher} has a negative answer when $M$ is not Riemannian, for example if $M=P_n$ is a regular polygon equipped with the Euclidean metric.

\begin{theorem-secondary}
Let $l\ge 1$, $n\geq 4l + 2$, and suppose $n = q(2l+1)$. Then for any $\varepsilon>0$, $z\ge q$, and $s_{n,l}<r<t_{n,l}$, there is an $\varepsilon$-dense finite subset $X\subseteq P_n$ such that $\vrc{X}{r} \simeq \bigvee^{z-1} S^{2l}$.
\end{theorem-secondary}

A version of Theorem~\ref{thm:secondary} is also true when the underlying polygonal metric space $P_n$ is instead replaced with an ellipse of sufficiently small eccentricity~\cite[Theorem~7.2]{AAR}, but only in the case $l=1$ (which corresponds to wedges of 2-dimensional spheres). The polygons provide examples where arbitrarily dense subsets of a metric space can have Vietoris--Rips complexes that obtain the homotopy type of a wedge sum of an arbitrary number of spheres, where the spheres are now of an arbitrarily high (even) dimension. That is, the polygons are examples where the homology of arbitrarily dense finite subsets is ``unstable'' in arbitrarily high homological dimensions.

A key contribution of our paper is the development of metric cyclic graphs (Section~\ref{sec:metric}), which sharpen the theory of infinite cyclic graphs in the metric setting. Metric cyclic graphs are more broadly applicable beyond the primary example of regular polygons in this paper; for example metric cyclic graphs generalize the framework used in the case of ellipses~\cite{AAR}, and they also potentially relate to some of the future work discussed in~\cite[Section~12]{virk20171}.

The remainder of our paper is organized as follows. In Section~\ref{sec:preliminaries} we introduce preliminaries and notation, and in Section~\ref{sec:cyclic} we review cyclic graphs. In Section~\ref{sec:metric} we introduce our main tool, metric cyclic graphs. We provide the necessary geometric lemmas about regular polygons in Section~\ref{sec:geometric}, and in Section~\ref{sec:vr} we give our main results about the homotopy types and persistent homology of Vietoris--Rips complexes of regular polygons. In Section~\ref{sec:subsets} we give more detailed results for finite subsets from a polygon. As an application of our results, in Section~\ref{sec:GH} we give a lower bound on the Gromov--Hausdorff distance between the regular polygon and the circle.

\section{Preliminaries and notation}\label{sec:preliminaries}

\subsection*{Notation for the plane}

We denote the length of a Euclidean vector $x\in\R^2$ by $\|x\|$, and therefore the Euclidean distance between two points $x,y\in\R^2$ is $\|x-y\|$.
We let $B(x,r)$ be the open ball in $\R^2$ with center $x\in \R^2$ and radius $r>0$. Likewise, we let $\overline{B(x,r)}$ be the corresponding closed ball. 

\subsection*{Notation for the circle}

We write $S^1$ to mean the unit circle in the plane, centered at the origin.

Next we define the \emph{normalized counterclockwise distance} $\ds \colon S^1 \times S^1 \r [0,1)$ by writing $\ds(x,y)$ to denote $\theta/2\pi$, where $\theta$ is the counterclockwise angular distance from $x$ to $y$ (i.e.\ the geodesic distance on $S^1$). Then we identify $S^1$ with the interval $[0,1)$. More specifically, we fix a basepoint $x_0 \in S^1$ and identify $x \in S^1$ with $\ds(x_0,x)\in[0,1)$.

We use a ternary relation to describe the ordering of three points on $S^1$, writing $x_1 \preceq x_2 \preceq x_3\preceq x_1$ when $x_1,x_2,x_3$ appear on $S^1$ in this counterclockwise order, allowing equality. We similarly write $x_1\preceq x_2\preceq \ldots\preceq x_s\preceq x_1$ to denote that $x_1,\ldots,x_s$ appear in $S^1$ in this counterclockwise order. We may replace $\ldots\preceq x_i\preceq x_{i+1}\preceq\ldots$ with $\ldots\prec x_i\prec x_{i+1}\prec\ldots$ when in addition we have $x_i\neq x_{i+1}$.

For $x,y\in S^1$ and for $x\neq y$ we denote the closed counterclockwise arc from $x$ to $y$ by $[x,y]_{S^1}=\{z\in S^1 ~|~x \preceq z \preceq y \preceq x\}$. Open and half-open arcs are defined similarly and denoted $(x,y)_{S^1}$, $(x,y]_{S^1}$, or $[x,y)_{S^1}$.

\subsection*{Notation for regular polygons}

For integer $n \ge 3$, we write $P_n$ to denote the boundary of the regular polygon of $n$ sides, inscribed in $S^1$. We equip $P_n$ with the Euclidean metric of the plane.

Let $\dpn(x,y)$ be equal to the fractional number of edges of $P_n$ contained in $[x,y]_{P_n}$ multiplied by $1/n$. In this sense, $\dpn(x,y)$ denotes the normalized counterclockwise geodesic distance between two points $x,y \in P_n$. After fixing some arbitrary basepoint $x_0\in P_n$, this allows us to identify $P_n$ with $[0,1)$ by identifying $x\in P_n$ with $\dpn(x_0,x)$. Since we have identified both $S^1$ and $P_n$ with $[0,1)$, this gives us a fixed homeomorphism between $S^1$ and $P_n$.

As in the $S^1$ case, we use a ternary relation to describe the ordering of three points on $P_n$, writing $x_1 \preceq x_2 \preceq x_3\preceq x_1$ when $x_1,x_2,x_3$ appear on $P_n$ in this counterclockwise order, allowing equality. We also use $\prec$ notation to disallow equality as in the $S^1$ case.

For $x,y\in P_n$ and for $x\neq y$ we denote the closed, open, or half-open counterclockwise arcs from $x$ to $y$ by $[x,y]_{P_n}$, $(x,y)_{P_n}$, or $(x,y]_{P_n}$, respectively.

\subsection*{Graphs and directed graphs}
A directed graph is a pair $G = (V, E)$ with $V$ the set of vertices and $E \subseteq V \times V$ the set of directed edges, where we require that there are no loops and that no edges are oriented in both directions. We also denote the set of vertices by $V(G)$, and the directed edge $(u, w)$ will also be denoted $u \to w$. A homomorphism of directed graphs
$h\colon G \to \tilde{G}$ is a vertex map such that for every edge $u\to w$ in $G$, either $h(u) = h(w)$ or there is an edge
$h(u) \to h(w)$ in $\tilde{G}$. For a vertex $u\in V$ we define the out- and in-neighborhoods
\[N^+(G,u) =\{w~|~u\diredge w\}, \quad N^-(G,u)=\{w~|~w\diredge u\}, \]
as well as their respective closed versions
\[ N^+[G,u]=N^+(G,u)\cup{u}, \quad N^-[G,u]=N^-(G,u)\cup{u}. \]
An undirected graph is a graph in which the orientations on the edges are omitted. Given a graph $G= (V,E)$ (directed or undirected) and a subset $S \subseteq V$, we write $G[S]$ to denote the induced subgraph with vertex set $S$.

\subsection*{Topological spaces} See~\cite{armstrong2013basic,Hatcher} for background on topological spaces. If $X$ and $Z$ are homotopy equivalent topological spaces, then we write $X\simeq Z$. The $i$-fold wedge sum of a space $X$ with itself is denoted $\bigvee^i X$. 
For $X$ a topological space and $G$ an abelian group, we let $H_i(X;G)$ denote the $i$-dimensional homology of $X$ computed with coefficients in the abelian group $G$. Roughly speaking, $H_i(X;G)$ measures the number of ``$i$-dimensional holes" in $X$.

\subsection*{Simplicial complexes}
 A \emph{geometric $k$-simplex} is the convex hull of $k + 1$ affinely independent points in Euclidean space. A  \emph{simplicial complex} is a collection of geometric simplices such that
 \begin{itemize}
 \item Every face of a simplex in $K$ also belongs to $K$.
 \item For any two simplices $\sigma_1$ and $\sigma_2$ in $K$, if $\sigma_1 \cap \sigma_2 \neq \emptyset$, then $\sigma_1 \cap \sigma_2$ is
a common face of both $\sigma_1$ and $\sigma_2$.
 \end{itemize}
 
An abstract simplicial complex is a pair $(V, \Sigma)$
with $V$ a finite set of vertices, and where $\Sigma$ is a subset (called the simplices) of the collection of all non-empty subsets of $V$, satisfying the condition that if $\sigma \in \Sigma$ and $\emptyset\neq \tau \subseteq \sigma$, then $\tau \in \Sigma$. An abstract simplicial complex can always be associated to a Euclidean simplicial complex, which we call a \emph{geometric realization} of the complex \cite{munkres-book}. We do not distinguish between an abstract simplicial complex and its geometric realization.

For $G$ an undirected graph, the clique complex $\cl(G)$ is the simplicial complex with vertex set $V(G)$ and with faces determined by all cliques (complete subgraphs) of $G$.

\subsection*{Vietoris--Rips simplicial complexes}
For $X$ a metric space and $r>0$ a scale parameter, the \emph{Vietoris--Rips} complex $\vrcless{X}{r}$ (resp.\ $\vrcleq{X}{r}$) is the simplicial complex with vertex set $X$, where a finite subset $\sigma\subseteq X$ is a face if and only if the diameter of $\sigma$ is less than $r$ (resp.\ at most $r$)~\cite{Hausmann1995,Vietoris27}. Note that a Vietoris--Rips simplicial complex is the clique complex of its 1-skeleton.

\begin{figure}[h]
\centering
\includegraphics[width=0.35\textwidth]{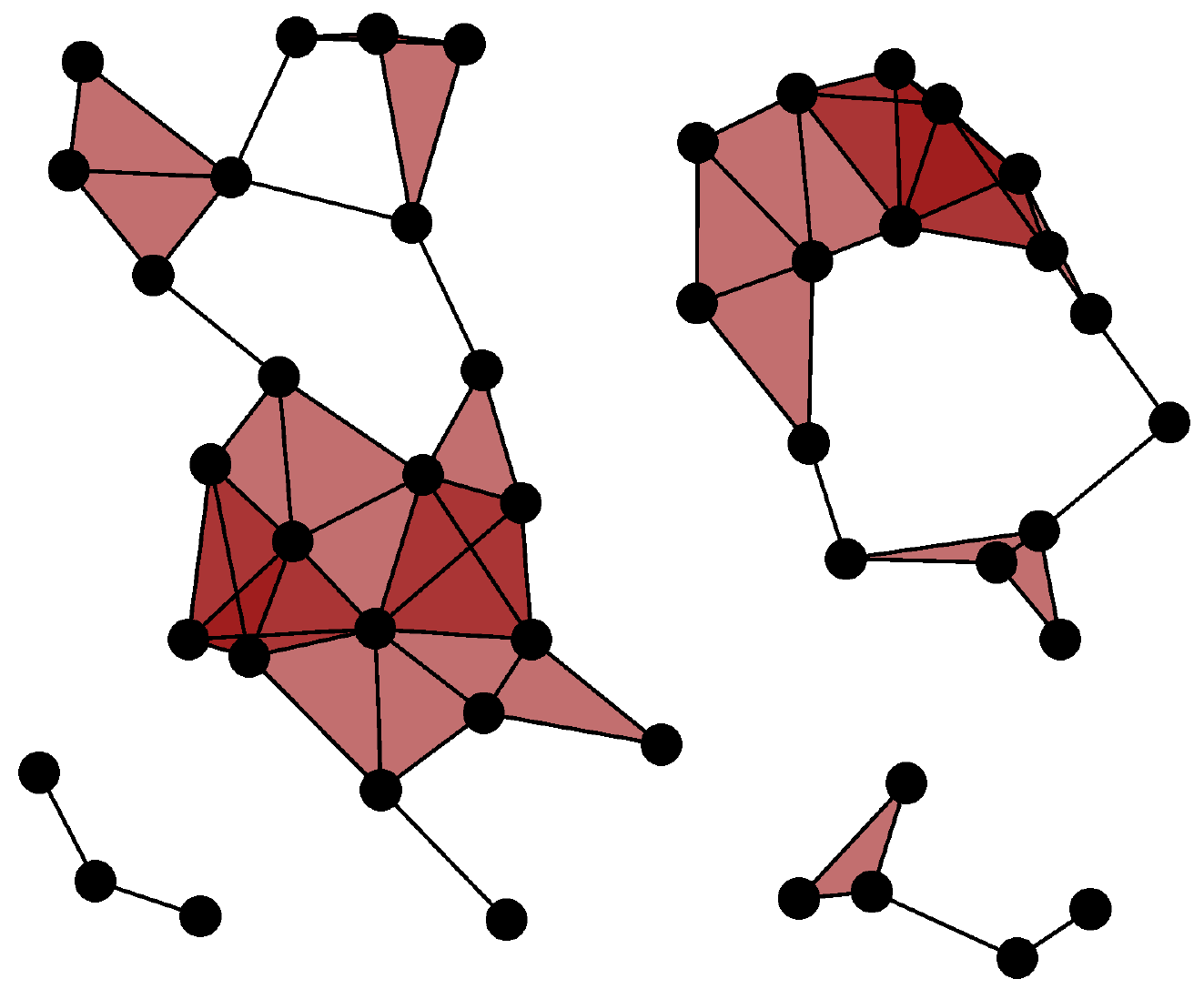}
\caption{A Vietoris--Rips complex of a finite set of points in the plane.}
\label{fig:Rips}
\end{figure}

\subsection*{Persistent homology}

Given an increasing sequence (i.e.\ a filtration) of topological spaces
\[ X_0 \hookrightarrow X_1 \hookrightarrow \ldots \hookrightarrow X_{k-1} \hookrightarrow X_k, \]
one can apply the $i$-dimensional homology functor with coefficients in a field $\F$ to obtain a sequence of vector spaces and linear maps
\[ H_i(X_0;\F) \to H_i(X_1;\F) \to \ldots \to H_i(X_{k-1};\F) \to H_i(X_k;\F).\]
The persistent homology of this sequence will be a multiset of intervals in which the start and end of each interval parametrizes, roughly speaking, the birth and death time of a topological feature in this filtered topological space~\cite{EdelsbrunnerHarer}.

For example, let $X$ be a metric space, and consider the Vietoris--Rips complexes $\vrc{X}{r}$ over a finite set of $r$-values $r_0 < r_1 < \ldots < r_{k-1} < r_k$. We obtain an increasing sequence of topological spaces 
\[ \vrc{X}{r_0} \hookrightarrow \vrc{X}{r_1} \hookrightarrow \ldots \hookrightarrow \vrc{X}{r_{k-1}} \hookrightarrow \vrc{X}{r_k}, \]
whose persistent homology describes how the shape of $\vrc{X}{r}$ changes as $r$ increases. If $X$ is a sample from some unknown underling metric space $M$, then the persistent homological features in $\vrc{X}{r}$ are often taken as an estimate of the homology of $M$, whereas the remaining short-lived features are often regarded as noise~\cite{Carlsson2009,ChazalOudot2008}.

Instead of selecting a finite number of scale parameters $r_0 < r_1 < \ldots < r_{k-1} < r_k$, one can instead vary $r$ over the entire interval $(0,\infty)$. Though it is more subtle, the  theory of persistent homology over the reals also exists (see for example~\cite{chazal2016structure}). The persistent homology diagrams of $\vrc{P_n}{r}$ which we describe, for example in Corollary~\ref{cor:main}, are over all real scale parameters $r\in(0,\infty)$.

\section{Cyclic graphs}\label{sec:cyclic}

In this section we review the basic theory of cyclic graphs. Cyclic graphs are of central importance to our results---we have a clear understanding of the homotopy type of simplicial complex $\vrc{P_n}{r}$ whenever $\vr{P_n}{r}$ is a cyclic graph. Specifically, the Vietoris-Rips complex will be homotopy equivalent to either an odd sphere or a wedge sum of even spheres of the same dimension. This theory in this section was developed and used in~\cite{Adamaszek2013,AA-VRS1,AAFPP-J,AAM,AAR,Reddy}.
 
\begin{definition}
\label{defn:cyclic-graph}
Let $G=(V,E)$ be a directed graph, where $V$ is a equipped with a fixed injective map into $S^1$. Therefore $V$ inherits a counterclockwise ordering from $S^1$, along with the subspace topology. Then $G$ is said to be \emph{cyclic} if whenever there is a directed edge $u_0 \diredge u_1$, there are also edges $u_0 \diredge w \diredge u_1$ for all $u_0 \prec w \prec u_1 \prec u_0$. 
\end{definition}

The Vietoris--Rips graph of a circle, of an ellipse of small eccentricity, or of a regular regular polygon $P_n$ (so long as the scale is sufficiently small) are examples of cyclic graphs with an infinite number of vertices. However, we will introduce the theory in the (simpler) finite case first.

\subsection{Finite cyclic graphs}

We call a cyclic graph $G$ \emph{finite} if its underlying vertex set $V$ is finite, and \emph{infinite} otherwise. An important family of finite cyclic graphs are the regular cyclic graphs, defined as follows.

\begin{definition}
For integers $n$ and $k$ with $0 \leq k < \frac{1}{2} n$, the \emph{regular} cyclic graph $\cnk{n}{k}$ has vertex set $\{0,\ldots,n-1\}$ and edges $i \diredge (i+s) \mod n$ for all $i = 0, \ldots, n-1$ and $s = 1, \ldots, k$.
\end{definition}

Every finite cyclic graph naturally gives rise to a dynamical system.

\begin{definition} \label{defn:cyc-dyn-fin}
Let $G$ be a finite cyclic graph with vertex set $V$. The associated \emph{finite cyclic dynamical system} is generated by the map $f: V \to V$ given by writing 
\[f(u):= \argmax \{\ds(u,w)~|~w \in N^+[G,u]\}.\]
Note that $f$ is well-defined (i.e.\ $f(u)$ is a singleton in $V$) by the finiteness of $V$. We often call $f(u)$ the \emph{counterclockwise-most} vertex of $N^+[G,u]$.
\end{definition}

A vertex $u \in V$ is \emph{periodic} if $f^i(u) = u$ for some $i \ge 1$. If $u$ is periodic then we refer to $\{f^i(u) \ | \ i \ge 0 \}$ as a \emph{periodic orbit}, whose length $\olen$ is the smallest integer $i \ge 1$ such that $f^i(u) = u$. The \emph{winding number} of a periodic orbit is $\sum_{i = 0}^{\olen-1}\ds(f^i(u),f^{i+1}(u))$, the number of times that this periodic orbit ``winds around" the graph.

The next lemma appeared as~\cite[Lemma 2.3]{AAM} without proof; we provide a proof here for completeness.

{
\makeatletter
\tagsleft@false 

\begin{lemma}
\label{lem:fin-cyc-wf-const}
Let $G = (V,E)$ be a finite cyclic graph. Then all the periodic orbits of $G$ have the same length $\olen$ and the same winding number $\omega$.
\end{lemma}
\begin{proof}[Proof of Lemma~\ref{lem:fin-cyc-wf-const}]
First note that for any $u,w \in V$,
\begin{equation}
u\prec w \preceq f(u) \prec u \quad\mbox{implies}\quad u\prec w\preceq f(u)\preceq f(w)\preceq f^2(u) \prec f(u). \tag{monotonicity}
\end{equation}
To see this, let $u\prec w \preceq f(u) \prec u$ in $V$. 
Towards a contradiction, suppose $f(w)\prec f(u)\preceq f^2(u) \preceq f(w)$. But then we have $u \prec w \preceq f(w) \prec f(u) \preceq u$. Because $G$ is cyclic, the presence of the edge $u \r f(u)$ would imply the presence of the edge $w \r f(u)$. This contradicts the strict relation $f(w)\neq f(u)$ in the ordering $f(w)\prec f(u)\preceq f^2(u) \preceq f(w)$.

Next let $A= \{u,f(u),f^2(u),\ldots, f^{\olen}(u)\}$ be a periodic orbit of length $\olen$. Let $B$ be a distinct periodic orbit. We claim that by relabeling terms if necessary, we can pick a vertex $w\in B$ such that $u \prec w \prec f(u) \preceq u$. To see this, suppose towards a contradiction that no such $w \in B$ exists. Then there must be some $w \in B$ such that $w \prec u \preceq f(u) \prec f(w) \prec w$, where the strict relation $w\neq u$ holds because $B$ is distinct from $A$. But because $G$ is cyclic, this contradicts monotonicity above.

Let $w\in B$ be such that $u \prec w \prec f(u) \preceq u$. By the monotonicity observation, we have $f(u) \prec f(w) \prec f^2(u) \preceq f(u)$. Here the strict relations hold because $B$ and $A$ are distinct orbits.
By induction, we get that $u = f^\olen(u) \prec f^\olen(w) \prec f^{\olen+1}(u) = f(u) \preceq u$.

Now there are three cases: either $f^\olen(w) = w$, or $u \prec f^\olen(w) \prec w \prec f(u) \preceq u$, or $u \prec w \prec f^\olen(w) \prec f(u) \preceq u$. We show that the last two cases cannot occur.

Towards a contradiction, suppose $u \prec f^\olen(w) \prec w \prec f(u) \preceq u$. Then by repeated application of $f^\ell$, we get the following sequence of inequalities: $u\preceq \ldots \preceq f^{3\olen}(w) \preceq f^{2\olen}(w) \preceq f^\olen(w) \preceq w \prec u$. Because $V$ is finite and the sequence is ``bounded below" by $u$ (which is not in the orbit of $w$), there must exist some nonnegative integer $k$ such that $f^{k\olen}(w) = f^{(k+1)\olen}(w)$. But then $\{f^{k\olen}(w),f^{k\olen+1}(w),\ldots, f^{k\olen+\olen}(w)\}$ forms a periodic orbit that never returns to $w$, contradicting the periodicity of $w$. A similar argument shows that $u \prec w \prec f^\olen(w) \prec f(u) \preceq u$ also cannot occur, and so we must have $f^\olen(w) = w$. Thus the orbit of $w$ also has length $\olen$. 

The preceding work shows that the orbits of $u$ and $w$ are interleaved and have the same length. It follows that they have the same winding number.
\end{proof}

} 

By virtue of Lemma~\ref{lem:fin-cyc-wf-const}, we can define the \emph{winding fraction} of a finite cyclic graph $G$ as follows:

\begin{definition}
If $G$ is a finite cyclic graph in which periodic orbits have length $\olen$ and winding number $\omega$, then we define the \emph{winding fraction} of $G$ to be $\wf(G) = \frac{\omega}{\olen}$.
\end{definition}
Intuitively, the winding fraction measures the fractional amount a periodic vertex ``wraps around the cyclic graph" in a single step of the cyclic dynamical system.

We denote by $P$ the number of periodic orbits in a given finite cyclic graph. A finite cyclic graph has at least one periodic orbit since $V$ is finite. The following proposition is from~\cite[Propositions~4.2]{AAR}.

\begin{proposition}\label{prop:vrchtpy}
If $G$ is a finite cyclic graph, then
\[ \cl(G)\simeq\begin{cases}
S^{2l+1} & \mbox{if }\frac{l}{2l+1}<\wf(G)<\frac{l+1}{2l+3}\mbox{ for some }l\in\N,\\
\bigvee^{\numorb-1}S^{2l} & \mbox{if }\wf(G)=\frac{l}{2l+1}.
\end{cases} \]
\end{proposition}

This proposition, for example, allows us to determine the homotopy types of the clique complexes of the cyclic graphs in Figure~\ref{fig:cyclic}.

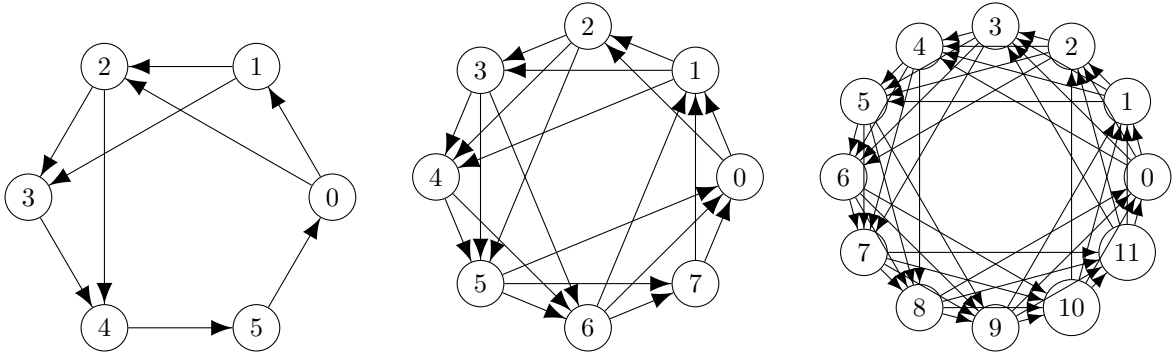
\begin{figure}[h]
\begin{center}
\begin{tikzpicture}[
		decoration={
			markings,
			mark=at position 1 with {\arrow[scale=2,black]{latex}};
		},every node/.style={draw,circle}
		]
		\def \n {5}
		\def \radius {3cm}
		\def \margin {8} 
		\foreach \s in {0,...,5}
		{\draw (60*\s: 2cm) node(\s){\s};}
		\draw [postaction=decorate] (0) -- (1);
		\draw [postaction=decorate] (0) -- (2);
		\draw [postaction=decorate] (1) -- (2);
		\draw [postaction=decorate] (1) -- (3);
		\draw [postaction=decorate] (2) -- (3);
		\draw [postaction=decorate] (2) -- (4);
		\draw [postaction=decorate] (3) -- (4);
		\draw [postaction=decorate] (4) -- (5);
		\draw [postaction=decorate] (5) -- (0);
		\end{tikzpicture}
		\hspace{5mm}
		\begin{tikzpicture}[
		decoration={
			markings,
			mark=at position 1 with {\arrow[scale=2,black]{latex}};
		},every node/.style={draw,circle}
		]
		\def \n {5}
		\def \radius {3cm}
		\def \margin {8} 
		\foreach \s in {0,...,7}
		{\draw (45*\s: 2cm) node(\s){\s};}
		\draw [postaction=decorate] (0) -- (1);
		\draw [postaction=decorate] (0) -- (2);
		\draw [postaction=decorate] (1) -- (2);
		\draw [postaction=decorate] (1) -- (3);
		\draw [postaction=decorate] (1) -- (4);
		\draw [postaction=decorate] (2) -- (3);
		\draw [postaction=decorate] (2) -- (4);
		\draw [postaction=decorate] (2) -- (5);
		\draw [postaction=decorate] (3) -- (4);
		\draw [postaction=decorate] (3) -- (5);
		\draw [postaction=decorate] (3) -- (6);
		\draw [postaction=decorate] (4) -- (5);
		\draw [postaction=decorate] (4) -- (6);
		\draw [postaction=decorate] (5) -- (6);
		\draw [postaction=decorate] (5) -- (7);
		\draw [postaction=decorate] (5) -- (0);
		\draw [postaction=decorate] (6) -- (7);
		\draw [postaction=decorate] (6) -- (0);
		\draw [postaction=decorate] (6) -- (1);
		\draw [postaction=decorate] (7) -- (0);
		\draw [postaction=decorate] (7) -- (1);
		\end{tikzpicture}
		\hspace{5mm}
		\begin{tikzpicture}[
		decoration={
			markings,
			mark=at position 1 with {\arrow[scale=1.5,black]{latex}};
		},every node/.style={draw,circle}
		]
		\def \n {5}
		\def \radius {3cm}
		\def \margin {8} 
		
		\foreach \s in {0,...,11}
		{\draw (30*\s: 2cm) node(\s){\s};}
		\draw [postaction=decorate] (0) -- (1);
		\draw [postaction=decorate] (1) -- (2);
		\draw [postaction=decorate] (2) -- (3);
		\draw [postaction=decorate] (3) -- (4);
		\draw [postaction=decorate] (4) -- (5);
		\draw [postaction=decorate] (5) -- (6);
		\draw [postaction=decorate] (6) -- (7);
		\draw [postaction=decorate] (7) -- (8);
		\draw [postaction=decorate] (8) -- (9);
		\draw [postaction=decorate] (9) -- (10);
		\draw [postaction=decorate] (10) -- (11);
		\draw [postaction=decorate] (11) -- (0);
		\draw [postaction=decorate] (0) -- (2);
		\draw [postaction=decorate] (1) -- (3);
		\draw [postaction=decorate] (2) -- (4);
		\draw [postaction=decorate] (3) -- (5);
		\draw [postaction=decorate] (4) -- (6);
		\draw [postaction=decorate] (5) -- (7);
		\draw [postaction=decorate] (6) -- (8);
		\draw [postaction=decorate] (7) -- (9);
		\draw [postaction=decorate] (8) -- (10);
		\draw [postaction=decorate] (9) -- (11);
		\draw [postaction=decorate] (10) -- (0);
		\draw [postaction=decorate] (11) -- (1);
		\draw [postaction=decorate] (0) -- (3);
		\draw [postaction=decorate] (1) -- (4);
		\draw [postaction=decorate] (2) -- (5);
		\draw [postaction=decorate] (3) -- (6);
		\draw [postaction=decorate] (4) -- (7);
		\draw [postaction=decorate] (5) -- (8);
		\draw [postaction=decorate] (6) -- (9);
		\draw [postaction=decorate] (7) -- (10);
		\draw [postaction=decorate] (8) -- (11);
		\draw [postaction=decorate] (9) -- (0);
		\draw [postaction=decorate] (10) -- (1);
		\draw [postaction=decorate] (11) -- (2);
		\draw [postaction=decorate] (0) -- (4);
		\draw [postaction=decorate] (1) -- (5);
		\draw [postaction=decorate] (2) -- (6);
		\draw [postaction=decorate] (3) -- (7);
		\draw [postaction=decorate] (4) -- (8);
		\draw [postaction=decorate] (5) -- (9);
		\draw [postaction=decorate] (6) -- (10);
		\draw [postaction=decorate] (7) -- (11);
		\draw [postaction=decorate] (8) -- (0);
		\draw [postaction=decorate] (9) -- (1);
		\draw [postaction=decorate] (10) -- (2);
		\draw [postaction=decorate] (11) -- (3);
\end{tikzpicture}
\end{center}
\caption{(Left) A cyclic graph of winding fraction $\frac{1}{4}$ with $\cl(G)\simeq S^1$. Node 1 is a slow point, and Node 2 is periodic. (Middle) A cyclic graph of winding fraction $\frac{1}{3}$ with $P=2$ periodic orbits and $\cl(G)\simeq S^2$. (Right) The cyclic graph $\cnk{12}{4}$ has winding fraction $\frac{1}{3}$ and $P=4$ periodic orbits. The homotopy type of its clique complex is $\cl(\cnk{12}{4})\simeq \bigvee^3 S^2$.}
\label{fig:cyclic}
\end{figure}

\subsection{Infinite cyclic graphs}

The theory of infinite cyclic graphs is more nuanced than the theory for finite cyclic graphs. In this section, we generalize the above notions to the infinite case, as this will be of central importance to our work. We begin by defining a notion of winding fraction for infinite cyclic graphs.

\begin{definition}
\label{defn:inf-cyc}
Let $G=(V,E)$ be an infinite cyclic graph. We define its winding fraction as
\[\wf(G):= \sup \{\wf(G[W])~|~W \mbox{ is a finite subset of } V\}. \]
We say that a cyclic graph $G$ is \emph{generic} if $\frac{l}{2l+1}<\wf(G)<\frac{l+1}{2l+1}$ for some integer $l\in\N$, or if $\wf(G)=\frac{l}{2l+1}$ and the supremum in Definition~\ref{defn:inf-cyc} is not attained. Otherwise we say $G$ is \emph{singular}.
\end{definition}
The homotopy properties of the clique complexes of generic cyclic graphs $G$ are simpler than those of singular cyclic graphs $G$; see Theorem~\ref{thm:homotopy-generic}.

In Definition~\ref{defn:cyc-dyn-fin}, we defined a finite dynamical system 
which does not always have an analogous definition in the infinite case. Nevertheless, we define an analogue of a dynamical system for infinite cyclic graphs as follows. 

\begin{definition}
\label{defn:cyc-dyn-inf}
Let $G$ be a cyclic graph with vertex set $V$, and let $m\ge 1$ be an integer. Recall $V$ is equipped with an embedding onto a subset of $S^1$; we identify $V$ forevermore with its image under this embedding.
We define the \emph{winding number map} $\gas_m \colon V\to\R$ by
\[ \gas_m(u_0) := \sup\Biggl\{\sum_{i=0}^{m-1}\ds(u_i,u_{i+1})~\Bigg|~\mbox{ there exists }u_0\diredge u_1\diredge\ldots\diredge u_m\Biggr\}. \]
The winding number map is naturally associated with a \emph{furthest point map} $\fs_m\colon V\to S^1$ that is defined by $\fs_m(u):=(u+\gas_m(u))\mod 1$. 
\end{definition}

\begin{remark}[Comparison of Definitions~\ref{defn:cyc-dyn-fin} and~\ref{defn:cyc-dyn-inf}]
In Definition~\ref{defn:cyc-dyn-fin}, we obtained a map $V \r V$ because the underlying cyclic graph was finite. In contrast, the furthest point map in Definition~\ref{defn:cyc-dyn-inf} is generally a map $V \r S^1$. Intuitively, the $\gas_m$ and $\fs_m$ maps serve as a proxy for the dynamical system in Definition~\ref{defn:cyc-dyn-fin}.
\end{remark}

\begin{definition}
We say $u\in V$ is \emph{periodic} if there is some $i\ge 1$ such that $\gas_i(u)\in\N$ and the supremum defining $\gas_i(u)$ is achieved. As in the finite case, the \emph{length} $\olen$ of this periodic orbit is the smallest such integer $i$, and its \emph{winding number} is $\gas_\olen(u)$.

Recall that $\ds$ is normalized to be in the interval $[0,1)$, so to have $\gas_i(u) \in \N$ means that $\fs_i(u) = u$. To say that the supremum defining $\gas_i(u)$ is achieved means that we have directed edges $u \r u_1 \r u_2 \r \ldots \r u_i = u$.
\end{definition}

In the setting of infinite cyclic graphs, the topology of the vertex set gives rise to subtleties; this motivates the next definition.

\begin{definition}
\label{defn:closed-cts}
A cyclic graph $G$ is \emph{closed} if its vertex set $V$ is closed as a subset of $S^1$ (under the specified embedding). We say $G$ is \emph{continuous} if $\gas_m\colon V\to \R$ is continuous for all $m\ge 1$.
\end{definition}

The significance of the preceding definition deserves some explanation. First note that if $G= (V,E)$ is closed, then because a closed set contains its limit points, we obtain a furthest point map $\fs: V \r V$ (as opposed to the codomain being $S^1$). If $\gas$ is continuous, then for any $i, j \in \N$ we obtain the nice composition property $\fs_{i+j}(u) = \fs_j(\fs_i(u))$~\cite[Lemma 5.13]{AAR}. This does not hold in general if $\gas$ is not continuous, as in the following example.

\begin{example}\label{ex:closed-cts}
Consider the cyclic graph $G=(V,E)$ with $V=S^1$ and edge set given by the following rules, where we are using the identification of $S^1$ with $[0,1)$.
\begin{enumerate}
\item $(u,(u+t) \!\mod 1) \in E$ for all $u \in S^1$ and all $0 < t < \frac{1}{10}$,
\item $(u,w) \in E$ for all $u \in [\frac{1}{10},\frac{2}{10}]$ and all $w \in [\frac{2}{10},\frac{3}{10}]$.
\end{enumerate}
Note $\gas$ is not continuous.
Then $\gas_1(0) = \frac{1}{10}$, and $\gas_2(0) = \frac{2}{10}$. Thus $\fs_1(0)= \frac{1}{10}$, and $\fs_2(0) = \frac{2}{10}$. However, we have
\[ \fs_1(\fs_1(0)) = \fs_1(\tfrac{1}{10}) = \tfrac{3}{10} \neq \tfrac{2}{10} = \fs_2(0). \]
\end{example}

\subsection{Fast and slow points}

Let $G$ be a (possibly infinite) cyclic graph with \emph{rational} winding fraction $\tfrac{p}{q}$. If a vertex $v$ of $G$ is not periodic, then it can behave in one of two ways. Furthermore, this behavior plays a role in the characterization that we obtain of the homotopy types of clique complexes. We now describe these non-periodic behaviors.

\begin{definition}[Fast and slow points]
\label{defn:fast-slow} 
Let $G = (V,E)$ be a (possibly infinite) cyclic graph with a rational winding fraction $\tfrac{p}{q}$. Let $V_\numorb$ be the set of periodic points. 
We say that a non-periodic vertex $u\in V\setminus V_\numorb$ is \emph{fast} if $\gas_q(u)>p$, and \emph{slow} if $\gas_q(u)\le p$. Note that in the fast case we have $u\prec \fs_q(u) \prec \fs_{q+1}(u)\prec u$, i.e.\ the dynamical system causes $\fs_q(u)$ to overshoot $u$. In the slow case, we have $u\prec \fs_{q+1}(u)\prec \fs_q(u) \prec u$.
\end{definition}

We provide an illustration of a slow point in Figure~\ref{fig:cyclic}. 

Even if a vertex $u\in V$ is not periodic, the dynamical system might eventually push $u$ onto a periodic point, after which it follows a periodic trajectory. We describe this situation next. 

\begin{definition}[Permanently fast points]
\label{defn:fast-slow-perm}
Let $G = (V,E)$ be a (possibly infinite) cyclic graph with winding fraction $\tfrac{p}{q}$ and set of periodic points $V_\numorb$. 
We say that a fast (resp.\ slow) vertex $u \in V \setminus V_p$ \emph{achieves periodicity} if there exists $m \in \N$ such that the supremum defining $\gas_m(u)$ is achieved, $\fs_m(u)$ belongs to $V$, and $\fs_m(u)$ is periodic. Otherwise, we say that $u$ is \emph{permanently fast} (resp.\ \emph{permanently slow}). Two permanently fast points $u$ and $w$ are \emph{equivalent} if $\lim_{m\to\infty}f_{mq}(u)=\lim_{m\to\infty}f_{mq}(w)$.
\end{definition}

\begin{definition}[Invariant sets]
Let $G$ be a closed and continuous cyclic graph with rational winding fraction $\wf(G) = \frac{p}{q}$. An \emph{invariant set $I = E_0 \cup \ldots \cup E_{q-1}$ of permanently fast points} is a union of $q$ equivalence classes $E_0, \ldots, E_{q-1}$ of permanently fast points such that for all $u_i \in E_i$ and $j \ge 0$, we have $f_j(u_i) \in E_{i+j \mod q}$.
\end{definition}

By the results in~\cite{AAR}, the set of permanently fast points is partitioned into invariant sets.

\begin{lemma}[Partitioning into invariant sets; Lemma 5.7 of~\cite{AAR}]
\label{thm:partition-perm-fast}
Let $G = (V,E)$ be a closed and continuous cyclic graph with a rational winding fraction $\tfrac{p}{q}$. Then its set of permanently fast points are partitioned into a collection of invariant sets of permanently fast points.
\end{lemma}

The following theorem is a combination of Theorems~5.1, 5.2, and 5.9 from~\cite{AAR}. Recall that we write $P$ to denote the number of periodic orbits. We let $F$ denote the cardinal number of invariant sets of permanently fast points.

\begin{theorem}\label{thm:homotopy-generic}
Let $G$ be a cyclic graph.
\begin{itemize}
\item If $G$ is generic with $\frac{l}{2l+1}<\wf(G)\le\frac{l+1}{2l+3}$ for some $l\in\N$, then $\cl(G)\simeq S^{2l+1}$. 
\item Let $G$ be singular with $\wf(G)=\frac{l}{2l+1}.$ Then $\cl(G)\simeq\bigvee^\kappa S^{2l}$ for some cardinal number $\kappa$. If $G$ is furthermore closed and continuous with $P$ and $F$ finite, then $\cl(G)\simeq\bigvee^{P+F-1}S^{2l}.$
\end{itemize}
Moreover, if $\iota\colon G\hookrightarrow\tilde{G}$ is an inclusion of generic cyclic graphs with $V(G)=V(\tilde{G})$ and $\frac{l}{2l+1}<\wf(G)\le\wf(\tilde{G})\le\frac{l+1}{2l+3}$, then $\iota$ induces a homotopy equivalence of clique complexes.
\end{theorem}

\section{Metric cyclic graphs}\label{sec:metric}

In this section we study the properties of cyclic graphs which are also equipped with a compatible metric structure. Many of the properties in this section are generalizations of~\cite[Section~6]{AAR}, in which the cyclic graph has as its vertex set an ellipse of small eccentricity with the Euclidean metric.

\begin{definition}
\label{defn:metric-cyclic}
A cyclic graph $G$ is \emph{metric} when its vertex set $V=V(G)$ is equipped with a metric $\dv$ such that:
\begin{enumerate}
\item the topology induced by $\dv$ is the same as the subspace topology $V$ inherits from its fixed injection into $S^1$, and 
\item for all directed paths $u_0\to u_1\to u_2$ in $G$, we have $\dv(u_0,u_1) < \dv(u_0,u_2)$.
\end{enumerate}
\end{definition}

It is important to remark that the metric $\dv$ on $V$ need not be the same as the geodesic metric on $S^1$. Indeed, the example of interest in this paper is when $V$ is a subset of the regular polygon $P_n$ equipped with the Euclidean metric. We identify $V$ as a subset of $S^1$ via the fixed injective map from $P_n$ into $S^1$ as described in Section~\ref{sec:preliminaries}. However, the metric $\dv$ on $V$ is the restriction of the Euclidean metric on $P_n\subseteq\R^2$, which is different from the geodesic metric on $S^1$.

\subsection{Continuity in metric cyclic graphs}

The main goal of this section is to define a map that ``moves forward by distance $r$" in a metric cyclic graph (Definition~\ref{def:gv}), and to state its continuity properties. This map will allow us to define stars in metric cyclic graphs (Section~\ref{ss:metric-stars}), which are closely related to the winding fractions of Vietoris--Rips metric cyclic graphs (Section~\ref{ss:metric-VR}).

Recall from Definition~\ref{defn:cyc-dyn-inf} that the winding number map $\gas_m \colon V\to\R$ is defined by
\[ \gas_m(u_0) := \sup\Biggl\{\sum_{i=0}^{m-1}\ds(u_i,u_{i+1})~\Bigg|~\mbox{ there exists }u_0\diredge u_1\diredge\ldots\diredge u_m\Biggr\}, \]
and that the furthest point map $\fs_m\colon V\to S^1$ is defined by $\fs_m(u):=(u+\gas_m(u))\mod 1$. It will be convenient to have a version of $\gamma_1$ that is measured not using the geodesic metric on $S^1$, but instead using the metric $\dv$ on $V$.

\begin{definition} Let $G$ be a closed and continuous metric cyclic graph. We define the map $\gav \colon V\to \R$ by $\gav(u) := \dv(u,\fs_1(u))$. That is,
\[ \gav(u_0) := \sup\{\dv(u_0,u_1)~|~\mbox{ there exists }u_0\diredge u_1\}. \]
\end{definition}

Intuitively speaking, whereas $\gas_1(u)$ measures the length of the furthest step one can take from $u$ in the counterclockwise direction using the geodesic distance on $S^1$, the value $\gav(u)$ instead measures this distance using the metric on $V$.

\begin{definition} 
\label{def:subset-subscript}
Let $V$ be homeomorphic to a subset of $S^1$, and let $u_0,u_1 \in V$. Then we write $(u_0,u_1)_V$ to denote the set $\{w \in V~|~u_0 \prec w \prec u_1 \prec u_0\}$. 
\end{definition}

\begin{definition}[Functions based at a point]\label{def:gv}
Let $G$ be a closed and continuous metric cyclic graph with $V$ homeomorphic to $S^1$, and let $u \in V$. For any $r \in [0,\gav(u)]$, define $\gv_r(u)$ to be the unique point $w$ in the interval $[u,\fs_1(u)]_{V}$ such that $\dv(u,w) = r$. Also define a map $d_{V,u} \colon [u,\fs_1(u)]_V\to \R$ by setting $d_{V,u}(w) := \dv(u,w)$ for each $w \in [u,\fs_1(u)]_V$.
\end{definition}

A solution $w$ to the equation $\dv(u,\cdot) = r$ exists since $G$ is closed and continuous by the intermediate value theorem. Uniqueness of $w$ follows since $d_{V,u}$ is monotonically increasing on the interval $(u,\fs_1(u))_V$. 

\begin{lemma}\label{lem:g_r-cont-r}
Let $G$ be a closed and continuous metric cyclic graph with $V$ homeomorphic to $S^1$. For $u \in V$, the function $g_r\colon [0,\gav(u)]\to V$ is continuous.
\end{lemma}


\begin{proof}[Proof of Lemma~\ref{lem:g_r-cont-r}]
Let $r \in (0,\gav(u))$, and let $\varepsilon > 0$. Recall that the function $d_{V,u}\colon [u, \fs_1(u)]_V \to \R$ is continuous. Hence there exists some $\delta > 0$ such that $\dv(w,\gv_r(u)) < \delta$ for $w\in [u, \fs_1(u)]_V$ implies $|d_{V,u}(w) - r| < \varepsilon$. Let $w^- \in [u,\gv_r(u)]_V$ be a point with $\dv(w^-,\gv_r(u))\le\min\{\delta, \varepsilon\}$, and let $w^+ \in [\gv_r(u), \fs_1(u)]_V$ be such that $\dv(\gv_r(u),w^+)\le\min\{\delta, \varepsilon\}$. Indeed, this is possible since the arc $[u, \fs_1(u)]_V$ is a continuous curve through $\gv_r(u)$. Since the function $d_{V,u}$ is strictly increasing on $[u, \fs_1(u)]_V$, we have that $d_{V,u}(w^-) < r < d_{V,u}(w^+)$. It follows that for $|r - r'| < \min(r - d_{V,u}(w^-),d_{V,u}(w^+) - r)$, we have $\gv_{r'}(u) \in (w^-, w^+)_V$ and thus $d(\gv_r(u), \gv_{r'}(u)) < \varepsilon$.
\end{proof}

For $G$ a closed and continuous metric cyclic graph with vertex set $V$, let $\gamma=\inf_{u\in V}\{\gav(u)\}$.

\begin{lemma}\label{lem:g_r-cont-p}
Let $G$ be a closed and continuous metric cyclic graph with $V$ homeomorphic to $S^1$. For $r \in [0,\gamma]$, the function $\gv_r\colon V \to V$ is continuous.
\end{lemma}

\begin{proof}[Proof of Lemma~\ref{lem:g_r-cont-p}]
Fix $\varepsilon > 0$, let $u \in V$ be an arbitrary point. There exists some $\delta > 0$ such that $|r-r'| < \delta$ implies $|\gv_r(u)-\gv_{r'}(u)| < \varepsilon$ by Lemma~\ref{lem:g_r-cont-r}. We fix such a $\delta$ with $\delta<\varepsilon$. Let the point $u' \in V$ be such that $\dv(u',u) < \delta$. By the triangle inequality, we see that $\dv(u',\gv_{r-\delta}(u)) < r < \dv(u', \gv_{r+\delta}(u))$. By the continuity of $d_{V,u}\colon [u, \fs_1(u)]_V\to \R$, there is some point $w'\in[\gv_{r-\delta}(u), \gv_{r+\delta}(u)]_V$ (which implies $\dv(w',\gv_r(u))<\varepsilon$) satisfying $\dv(u',w') = r$. Hence the point $\gv_{r}(u') = w'$ satisfies $\dv(\gv_r(u),\gv_r(u')) < \varepsilon$, and therefore $\gv_r\colon G\to G$ is continuous.
\end{proof}

\begin{lemma}\label{lem:G}
Let $G$ be a closed and continuous metric cyclic graph with $V$ homeomorphic to $S^1$. The function $H \colon V \times [0,\gamma] \to V$ defined by $H(u,r) = \gv_r(u)$ is continuous.
\end{lemma}

\begin{proof}
Consider an arbitrary $u\in V$ and $r\in[0,\gamma]$, and restrict attention to a sufficiently small open neighborhood $U\subseteq V\times[0,\gamma]$, homeomorphic to $\R^2$, that contains $(u,r)$. Further restrict this neighborhood $U$ so that there exists some $w\in V$ with $H(U)\subseteq V\setminus\{w\}$. Parametrize $V\setminus\{w\}$ as a subset of $\R$, and consider the function $H|_U$ as a real-valued function on an open subset of the plane. The function $H|_U$ is monotonic in both of its variables, and so $H|_U$ is jointly continuous by~\cite[Proposition~1]{kruse1969joint}, Lemma~\ref{lem:g_r-cont-r}, and Lemma~\ref{lem:g_r-cont-p}. $H$ is jointly continuous since this argument holds for all $(u,r)\in V\times[0,\gamma]$.
\end{proof}

\subsection{Stars in metric cyclic graphs}\label{ss:metric-stars}

We describe stars in metric cyclic graphs, whose existence determines whether the winding fraction meets or exceeds a singular value of the form $\frac{l}{2l+1}$.

\begin{definition}
Let $G$ be a metric cyclic graph, let $l\ge1$, and let $r\in[0,\gamma]$. A \emph{$(2l+1)$-star of scale $r$} in $G$ is a directed path $u_0\to u_1\to\ldots\to u_{2l}\to u_{2l+1}=u_0$ in $G$ of winding number $l$ such that $\dv(u_i,u_{i+1})=r$ for all $i$. 
\end{definition}

\begin{lemma}\label{lem:star-unique}
Let $G$ be a closed and continuous metric cyclic graph with $V$ homeomorphic to $S^1$, let $l\ge1$, and let $u\in V$. Then there is at most one $(2l+1)$-star in $G$ containing $u_0$.
\end{lemma}

\begin{proof}
Suppose for a contradiction that there existed two distinct inscribed equilateral $(2l+1)$-pointed stars in $G$ of side lengths $r$ and $r'$; we may assume $r<r'$. Denote the vertex sets of the stars by $\{u,\gv_r(u),\dots (\gv_r)^{2l}(u)\}$ and $\{u,\gv_{r'}(u),\dots (\gv_{r'})^{2l}(u)\}$. This means we have $u \prec \gv_{r}(u) \prec \gv_{r'}(u)\prec u$, i.e.\ $(\gv_{r'})^{j-1}(u) \prec (\gv_{r})^j(u) \prec (\gv_{r'})^j(u)\prec (\gv_{r'})^{j-1}(u)$ for all $j$ by induction. Letting $j = 2l+1$, we see $(\gv_{r'})^{2l}(u) \neq u$, a contradiction.
\end{proof}

\subsection{Vietoris--Rips metric cyclic graphs}\label{ss:metric-VR}

An important class of metric cyclic graphs are those that are also Vietoris--Rips graphs.

\begin{definition}\label{def:rV}
Let $V$ be a metric space equipped with a fixed continuous homeomorphism to $S^1$. We define $r_V$ to be
\[ r_V=\sup\{r\ge0~|~\vr{V}{r'}\mbox{ is a cyclic graph for all }0<r'<r\}. \]
\end{definition}

Note that the value of $r_V$ is not affected by whether we use the the $<$ or $\le$ convention for the Vietoris--Rips graph. 

\begin{lemma}\label{lem:conn}
If $V$ is a metric space equipped with a fixed homeomorphism to $S^1$, then 
\[ r_V = \sup\{r > 0 : B(u,r')\cap V\mbox{ is connected for all }r'<r\mbox{ and }u\in V\}.\]
\end{lemma}

\begin{proof}
For the $\ge$ direction, let $r>0$, and suppose $B(u,r') \cap V$ is connected for all $r'<r$ and for all $u\in V$. Let $u_0,u_1 \in V$ with $u_0 \to u_1$ in $\vr{V}{r'}$. Since $B(u_0,r') \cap V$ is connected, for any $w \in V$ such that $u_0 \prec w \prec u_1 \prec u_0$, we have $u_0 \to w$. By the triangle inequality, we also have $w \to u_1$. Thus $\vr{V}{r'}$ is a cyclic by Definition~\ref{defn:cyclic-graph}, and furthermore a metric cyclic graph by Definition~\ref{defn:metric-cyclic}. This gives that $r_V\ge r$.

For the reverse direction, note that if any $B(p,r)\cap V$ is not connected, then $\vr{V}{r}$ is not cyclic. 
\end{proof}

\begin{lemma}\label{lem:metric-closed-cont}
If $V$ is a metric space equipped with a fixed homeomorphism to $S^1$, then $\vr{V}{r}$ is a closed and continuous cyclic graph for all $r<r_V$.
\end{lemma}

\begin{proof}
Let $r<r_V$. The definition of $r_V$ implies that $\vr{V}{r}$ is cyclic. Furthermore, $\vr{V}{r}$ is closed since $V$ is homeomorphic to $S^1$. Note that for all $m\ge 1$, the map $\gamma_m\colon V\to \R$ is continuous since the metric $d_V$ on $V$ is topologically equivalent to the geodesic metric on $S^1$ (see (1) in Definition~\ref{defn:metric-cyclic}). It follows that $\vr{V}{r}$ is a continuous cyclic graph.
\end{proof}

\begin{definition}\label{def:s-fns}
Let $V$ be a metric space equipped with a fixed homeomorphism to $S^1$, and fix $l\ge 1$. Let $V'\subseteq V$ be the set of all points $u\in V$ for which there is a $(2l+1)$-star in $\vr{V}{r}$ containing $u$ for some $r\in(0,r_V)$. Define the function $s_{2l+1}\colon V'\to \R$ by sending a point $u\in V'$ to the scale parameter $r$ of the (necessarily unique) corresponding $(2l+1)$-star.
\end{definition}

\begin{lemma}
\label{lem:s-cts}
The function $s_{2l+1}\colon V' \to (0,r_V)$ is continuous.
\end{lemma}

\begin{proof}
Let $u\in V'$ and let $r<r_V$. Since $\gv_r(u)$ is the unique solution $q$ to $d_{V,u}(q) = r$ on $[u,f_1(u)]_V$, and since $d_{V,u}$ is strictly increasing along this interval, we see that $\gv_r(u)$ is monotonic in $r$, that is, for $0 < r < r' < \gav(u)$ we have $u \prec \gv_r(u) \prec \gv_{r'}(u) \prec u$. It follows that for $|r - r'|$ sufficiently small, we have 
\begin{equation}\label{eq:monotonic}
(\gv_r)^{2l}(u) \prec (\gv_r)^{2l+1}(u) \prec (\gv_{r'})^{2l+1}(u) \prec (\gv_r)^{2l}(u).
\end{equation}

Let $u \in V'$ and let $r = s_{2l+1}(u)<r_V$. Let $\varepsilon>0$ be an arbitrarily small constant satisfying $0<r-\varepsilon$ and $r+\varepsilon<r_V$; by \eqref{eq:monotonic} we have
\[(\gv_r)^{2l}(u) \prec (\gv_{r-\varepsilon})^{2l+1}(u) \prec u \prec (\gv_{r+\varepsilon})^{2l+1}(u) \prec (\gv_r)^{2l}(u).\]
By continuity of $(\gv_r)^{2l+1}\colon V'\to V'$, for $\tu\in V'$ with $d(u,\tu)$ sufficiently small we have
\[(\gv_r)^{2l}(u) \prec (\gv_{r-\varepsilon})^{2l+1}(\tu) \prec \tu \prec (\gv_{r+\varepsilon})^{2l+1}(\tu) \prec (\gv_r)^{2l}(u).\]
The monotonicity and continuity of $(\gv_r)^{2l+1}$ then imply there exists some $r^*\in(r-\varepsilon,r+\varepsilon)$ with $(\gv_{r^*})^{2l+1}(\tu) = \tu$. Hence $s_{2l+1}(\tu)=r^*$ with $|s_{2l+1}(u)-s_{2l+1}(\tu)|=|r - r^*| < \varepsilon$, and so $s_{2l+1}$ is continuous.
\end{proof}

\begin{definition}\label{def:s,t}
Let $V$ be a metric space equipped with a fixed homeomorphism to $S^1$. We define $s_{V,l}$ and $t_{V,l}$ to be
\[s_{V,l}=\sup\Bigl\{r<r_V~\Big|~\wf(\vr{V}{r})<\frac{l}{2l+1}\Bigr\}
\quad\mbox{and}\quad
t_{V,l}=\sup\Bigl\{r<r_V~\Big|~\wf(\vr{V}{r})\le\frac{l}{2l+1}\Bigr\}.\]
\end{definition}
By definition we have $s_{V,l}\le t_{V,l}\le r_V$. Roughly speaking, $s_{V,l}$ is the smallest scale of an inscribed $(2l+1)$-star in $\vr{V}{r}$, and $t_{V,l}$ is the largest scale of such an inscribed star. Note that if $V$ is homeomorphic to $S^1$ then $t_{V,0}=0$. By Theorem~~\ref{thm:homotopy-generic}, we know that the values $s_{V,l}$ and $t_{V,l}$ are critical scale parameters where the homotopy type of $\vrc{V}{r}$ changes.

\section{Geometric lemmas for regular polygons}\label{sec:geometric}

We now specialize to the specific case when $V=P_n$ is a regular polygon with $n$ sides in the plane, equipped with the Euclidean metric. The first question we address is finding values of $r > 0$ for which $\vr{P_n}{r}$ is a cyclic graph (Section~\ref{ss:cyclic}). As usual, the topology of the Vietoris--Rips complex $\vrc{P_n}{r}$ is then completely determined whenever the winding fraction $\wf(\vr{P_n}{r})$ is generic. We therefore focus attention on the case when the winding fraction is singular. That is, we will be interested in characterizing the scale parameters for which equilateral $(2l+1)$-stars of some side length $r$ can be inscribed into $P_n$, which determine when the winding fraction of $\vr{P_n}{r}$ first reaches $\frac{l}{2l+1}$ (Section~\ref{ss:side-length}). Finally, we will be interested in the number of such stars in this singular regime (Section~\ref{ss:counting}). This will give us a count of the number of periodic orbits in the dynamical system on the Vietoris-Rips graph, which will allow us to invoke Theorem~\ref{thm:homotopy-generic} in proving the homotopy types of $\vrc{P_n}{r}$ in Theorem~\ref{thm:main-general}.

Up to a rigid isometry, we can assume that the vertices of the regular polygon $P_n$ are the $n$th roots of unity $1,\omega,\omega^2,\ldots,\omega^{n-1}$ in the complex plane, where $\omega=e^{2\pi i/n}$ (we are implicitly using the canonical identification between $\R^2$ and $\mathbb{C}$).
We will often refer to the ``corner" points $1,\omega,\omega^2,\ldots,\omega^{n-1}$ as \emph{vertices} of $P_n$, to distinguish them from the other points of $P_n$ (which are of the form $t\omega^j+(1-t)\omega^{j+1}$ for $0<t<1$).

\subsection{Cyclic graph regime}\label{ss:cyclic} In this subsection, we characterize the parameter values $r$ for which $\vr{P_n}{r}$ is a cyclic graph, and hence our machinery can be applied.

In the case when $V=P_n$, we use the symbol $r_n$ to denote $r_V=r_{P_n}$ (Definition~\ref{def:rV}). The following values for $r_n$ correspond to distances between vertices of $P_n$ and their projections onto the ``opposite" side of $P_n$. An illustration for $n = 5,6$ is provided in Figure~\ref{fig:polygon-rn}.

\begin{proposition}\label{prop:connected}
If $n\ge 4$, then \begin{equation*}
r_n = \begin{cases}
2\cos(\pi/n) & \mbox{ if } n \mbox{ is even} \\[10pt]
1+\tfrac{\cos(2\pi/n)}{\cos(\pi/n)} & \mbox{ if } n \mbox{ is odd.} \\
\end{cases}
\end{equation*}
It follows that $\vr{X}{r}$ is a closed and continuous metric cyclic graph for all $r<r_n$.
\end{proposition}

\begin{figure}
\begin{minipage}{0.6\textwidth}
\begin{center}
\captionsetup{width=.9\linewidth}
\def\svgwidth{\linewidth}
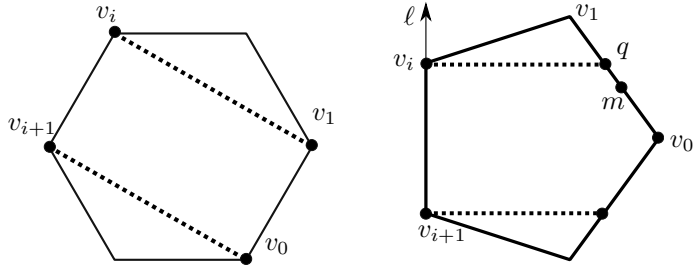
\end{center}
\end{minipage}
\caption{The values of $r_n$ for $n=5,6$ are given by the lengths of the dotted lines shown above. The labels correspond to the proof of Proposition~\ref{prop:connected} when $n$ is even (left) and when $n$ is odd (right). 
}
\label{fig:polygon-rn}
\end{figure}

\begin{proof}
By Lemma~\ref{lem:conn}, we have that
\[ r_n = \sup\{r > 0 : B(p,r')\cap P_n\mbox{ is connected for all }r'<r\mbox{ and }p\in P_n\}.\]
We first consider the case when $n$ is even (Figure~\ref{fig:polygon-rn} (left)). If $v_0$ is a vertex of $P_n$, then $B(v_0,r)\cap P_n$ is connected for all $r$. Otherwise, let $p\in (v_0,v_1)_{P_n}$ be a point between two adjacent vertices $v_0$ and $v_1$ in $P_n$. Let $(v_i,v_{i+1})_{P_n}$ be the edge opposite $(v_0,v_1)_{P_n}$, and let $\ell$ be the line between $v_i$ and $v_{i+1}$. Note that $\|p - \proj_\ell(p)\|=2\cos(\pi/n)$. It is not hard to check that $B(p,r)\cap P_n$ is connected for any $0<r\le2\cos(\pi/n)$, and that $B(p,2\cos(\pi/n)+\varepsilon)\cap P_n$ is not connected for any sufficiently small $\varepsilon>0$. Hence $r_n=2\cos(\pi/n)$. 

Now let $n$ be odd (Figure~\ref{fig:polygon-rn} (right)). Let $v_0$ and $v_1$ be two adjacent vertices in $P_n$, and let one of the two edges opposite $(v_0,v_1)_{P_n}$ be $(v_i,v_{i+1})_{P_n}$. Let $\ell$ be the line between $v_i$ and $v_{i+1}$. Let $q$ denote the unique point on $[v_0,v_1)_{P_n}$ such that $\proj_\ell(q) = v_i$. Let $m$ be the midpoint of $v_0$ and $v_1$. It is not hard to check that for $p\in[q,v_1]_{P_n}$ all balls $B(p,r)\cap P_n$ are connected for any $r\le\norm{q - v_i}$. Furthermore, for $p\in(m,q)_{P_n}$ we have that $B(p,r)\cap P_n$ is connected for any $0<r\le\norm{p - \proj_\ell(p)}$, that $B(p,\norm{p - \proj_\ell(p)}+\varepsilon)\cap P_n$ is not connected for any sufficiently small $\varepsilon>0$, and that $\norm{p - \proj_\ell(p)}>\norm{q-v_i}$. It follows that $r_n=\norm{q - v_i} = 1+\tfrac{\cos(2\pi/n)}{\cos(\pi/n)}$.

The fact that $\vr{P_n}{r}$ is a closed and continuous metric cyclic graph for all $r<r_n$ follows from Lemma~\ref{lem:metric-closed-cont}.
\end{proof}

\begin{remark}\label{rem:rn}
We note that $r_3=0$, i.e., that $\vr{P_3}{r}$ is not a cyclic graph for any $r>0$. For $n\ge 4$, we have that $\vr{P_n}{r}$ is a cyclic graph for all $r<r_n$, that $\vrless{P_n}{r_n}$ is a cyclic graph, and that $\vrleq{P_n}{r_n}$ is a cyclic graph if and only if $n$ is odd. In this paper we typically restrict to $r<r_n$ for the sake of simplicity. 
\end{remark}

\begin{remark}
Note that as $n \to \infty$ we have $r_n \to 2$, which is the diameter of the circle (with the Euclidean metric) in which the polygons are inscribed.  
\end{remark}

\subsection{Shape of the side length function}\label{ss:side-length}

Recall that Lemma~\ref{lem:star-unique} established uniqueness but not existence for stars inscribed at a particular point in a metric cyclic graph. Here we give sufficient conditions for existence in the case of $P_n$, which will guarantee that the side length function (Definition~\ref{def:s-fns}) is well-defined on all of $P_n$.

As a technical note, by Remark~\ref{rem:rn} the map $g_{r}\colon P_n\to P_n$ is easiest to define for $r \in [0,r_n)$. Nevertheless, by continuity we may extend our domain to $[0,r_n]$ by assigning $g_{r_n}(p) = \lim_{r\to r_n}g_r(p)$.

The next lemma shows that the dynamics at a vertex of $P_n$ is slower than at any other point. Recall that $\dpn$ gives the counterclockwise geodesic distance along $P_n$.

\begin{lemma}\label{lem:min-winding}
Let $0<r<r_n$, let $v \in P_n$ be a vertex, and let $p \in P_n$ be an arbitrary point. Then we have $\dpn(v,g_{r}(v)) \le \dpn(p,g_{r}(p))$.
\end{lemma}

\begin{proof}
The result is clear when $r\le \frac{1}{n}$, i.e.\ when $v$ and $g_r(v)$ lie on the same edge of $P_n$. Hence we may assume $r>\frac{1}{n}$.

Let $v_0$ be a vertex of $P_n$ and set $u_1 = g_{r}(v_0)$. Let $v_1$ be the unique vertex of $P_n$ that is in $(v_0,u_1)_{P_n}$ and is adjacent to $u_1$, and also set $u_0 = g_{r}^{-1}(v_1)$ (see Figure~\ref{fig:min-winding}.) Without loss of generality (by symmetry), let $p\in (u_0,v_0)_{P_n}$, and define $q=g_{r}(p)$. Note that $p$ and $q$ lie on edges of $P_n$ which are not parallel, so the edges intersect at a point $c$ with incident angle $\theta$. We then define $x = \|p-v_0\|$ and $y = \|q-v_1\|$. With this notation, we can write
\[ \dpn(p,q) = \begin{cases}
\frac{1}{n}\bigg(\frac{n}{2} - 2 + x + y \bigg) &\mbox{ if $n$ is even} \\
\frac{1}{n}\bigg(\frac{n+1}{2} - 2 + x + y \bigg) &\mbox{ if $n$ is odd.}
\end{cases} \]

\begin{figure}
\begin{minipage}{0.2\textwidth}
\def\svgwidth{0.7\linewidth}
\begingroup%
  \makeatletter%
  \providecommand\color[2][]{%
    \errmessage{(Inkscape) Color is used for the text in Inkscape, but the package 'color.sty' is not loaded}%
    \renewcommand\color[2][]{}%
  }%
  \providecommand\transparent[1]{%
    \errmessage{(Inkscape) Transparency is used (non-zero) for the text in Inkscape, but the package 'transparent.sty' is not loaded}%
    \renewcommand\transparent[1]{}%
  }%
  \providecommand\rotatebox[2]{#2}%
  \newcommand*\fsize{\dimexpr\f@size pt\relax}%
  \newcommand*\lineheight[1]{\fontsize{\fsize}{#1\fsize}\selectfont}%
  \ifx\svgwidth\undefined%
    \setlength{\unitlength}{383.97598267bp}%
    \ifx\svgscale\undefined%
      \relax%
    \else%
      \setlength{\unitlength}{\unitlength * \real{\svgscale}}%
    \fi%
  \else%
    \setlength{\unitlength}{\svgwidth}%
  \fi%
  \global\let\svgwidth\undefined%
  \global\let\svgscale\undefined%
  \makeatother%
  \begin{picture}(1,1.41767007)%
    \lineheight{1}%
    \setlength\tabcolsep{0pt}%
    \put(0,0){\includegraphics[width=\unitlength,page=1]{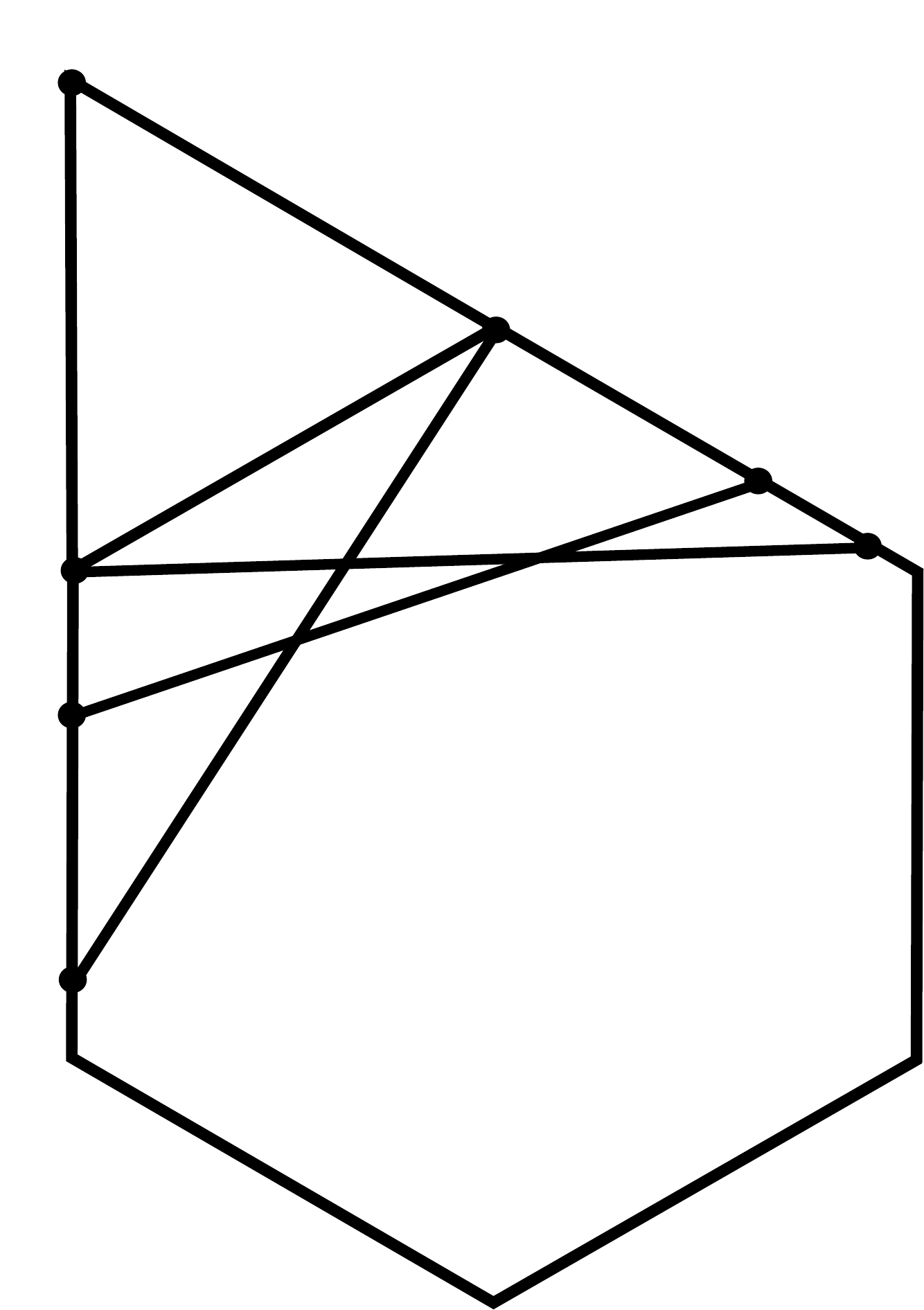}}%
    \put(0.05998365,1.3973212){\color[rgb]{0,0,0}\makebox(0,0)[lt]{\lineheight{1.25}\smash{\begin{tabular}[t]{l}$c$\end{tabular}}}}%
    \put(0.55319486,1.1045385){\color[rgb]{0,0,0}\makebox(0,0)[lt]{\lineheight{1.25}\smash{\begin{tabular}[t]{l}$v_0$\end{tabular}}}}%
    \put(-0.10093084,0.8097907){\color[rgb]{0,0,0}\makebox(0,0)[lt]{\lineheight{1.25}\smash{\begin{tabular}[t]{l}$v_1$\end{tabular}}}}%
    \put(-0.08282405,0.64080199){\color[rgb]{0,0,0}\makebox(0,0)[lt]{\lineheight{1.25}\smash{\begin{tabular}[t]{l}$q$\end{tabular}}}}%
    \put(0.84204759,0.93751476){\color[rgb]{0,0,0}\makebox(0,0)[lt]{\lineheight{1.25}\smash{\begin{tabular}[t]{l}$p$\end{tabular}}}}%
    \put(0.95601673,0.86284525){\color[rgb]{0,0,0}\makebox(0,0)[lt]{\lineheight{1.25}\smash{\begin{tabular}[t]{l}$u_0$\end{tabular}}}}%
    \put(-0.10093084,0.36177424){\color[rgb]{0,0,0}\makebox(0,0)[lt]{\lineheight{1.25}\smash{\begin{tabular}[t]{l}$u_1$\end{tabular}}}}%
  \end{picture}%
\endgroup%

\end{minipage}
\begin{minipage}{0.2\textwidth}
\def\svgwidth{\linewidth}
    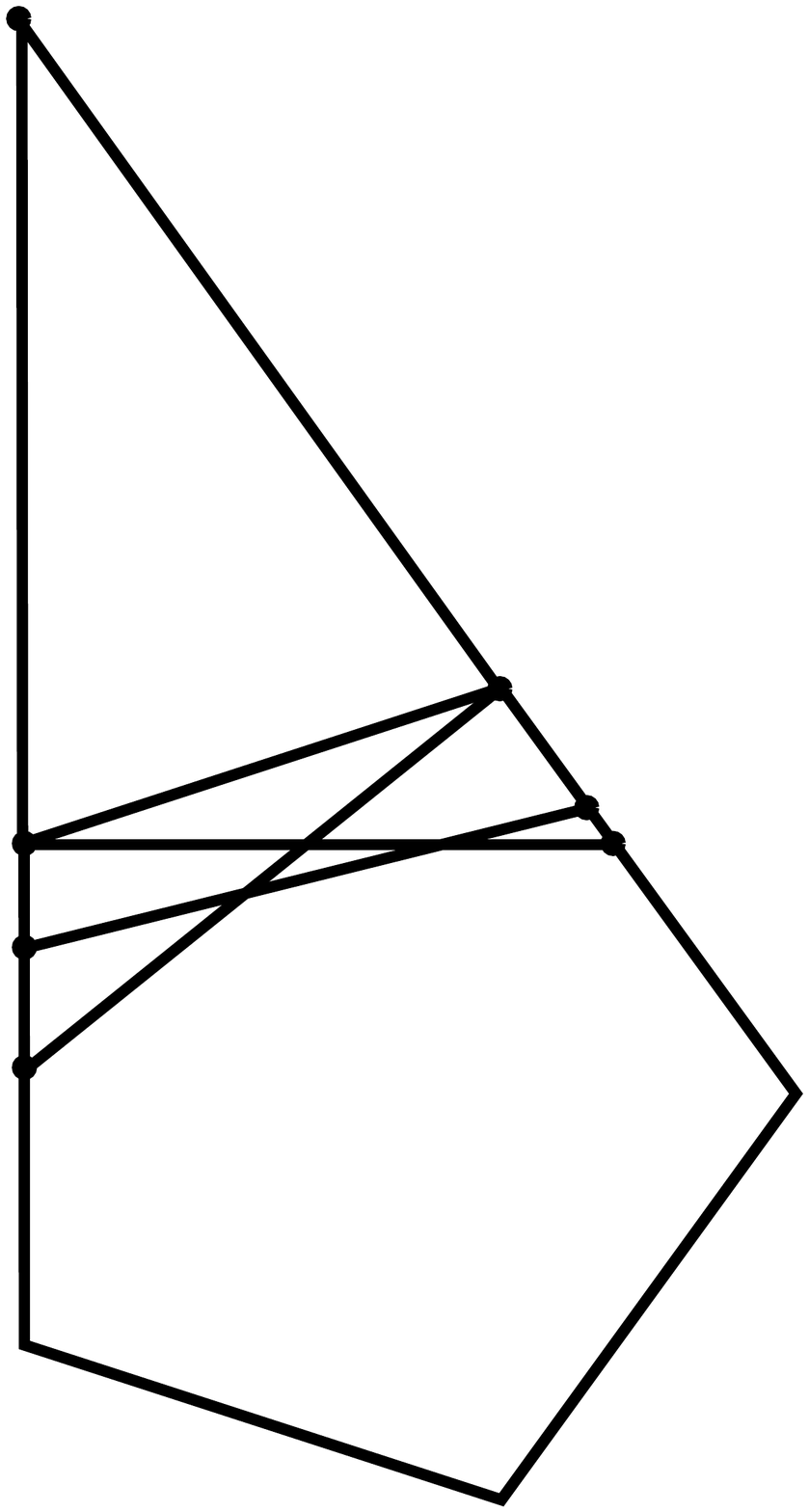
\end{minipage}
\hspace{1cm}
\begin{minipage}{0.2\textwidth}
\begin{tikzpicture}
	\newdimen\rad
	\rad=1.5cm
	\draw (0:\rad) \foreach \x in {72,144,...,359} {
		-- (\x:\rad)
	} -- cycle (90:\rad) ;
	\draw[<->] (0,-2) --(0,2);
	\draw[<->] (-2,0) --(2,0);
	\draw (-1.2,0.8795) --(0.865,0.8795);
	\filldraw (0.865,0.8795) circle (1pt);
	\filldraw (-1.2,0.8795) circle (1pt);
	\node at (1,1.1) {$u$};
	\node at (-1.3,1.1) {$v$};
\end{tikzpicture}
\end{minipage}
\caption{Visualization of the proof of Lemma~\ref{lem:min-winding} in the case that $n$ is odd (left) and $n$ is even (center). Picture for the proof of Lemma~\ref{lem:existence} when $n$ is odd (right).}
\label{fig:min-winding}
\end{figure}

We can prove the desired inequality by optimizing the function $f(x,y)\colon[0,r]^2\to\R$ defined by $f(x,y)=x+y$ subject to the law of cosines constraint $g(x,y) = x^2+y^2-2xy\cos\theta - r^2 = 0$. Using Lagrange multipliers, we compute $\nabla f = (1,1)$ and $\nabla g = (2x-2y\cos\theta, 2y - 2x\cos\theta)$, and note that the system $\nabla f = \lambda \nabla g$ is satisfied when $x = y$ or $\theta = \pi$. Since the latter case is impossible, we conclude that the only interior point extremum of $f$ is when $x = y$. Plugging this into the contraint $g(x_0,x_0) = 0$, we get $x_0 = r/\sqrt{2(1-\cos\theta)}$, so $f(x_0,x_0) = r\sqrt{2/(1-\cos\theta)}$. We compare this to the boundary value $r = f(0,r) < f(x_0,x_0)$ to conclude that $(x_0,x_0)$ is a global maximum of $f$, and that the minimum value of $f$ subject to the constraint $g$ is achieved at the boundary point $(x,y)=(0,r)$. It follows that $\dpn(v_1,g_{r}(v_1)) \le \dpn(p,g_{r}(p))$ holds for all $p \in P_n$.
\end{proof}

We use the above to prove the following existence result. Later, this will be used to deduce the winding fractions of $\vr{P_n}{r}$ as $r$ varies. 

\begin{lemma}\label{lem:existence}
For all $n\ge 4l+2$ except\footnote{When $n=6$ and $l=1$, we have $t_{P_n,l}=r_n$.} for $(n,l)=(6,1)$, there exists an inscribed equilateral $(2l+1)$-pointed star of radius $r< r_n$ at every basepoint $p \in P_n$.
\end{lemma}

\begin{proof}
Fix an integer $n \ge 4l+2$ and a point $p \in P_n$, and define the function $f_p\colon[0,r_n]\to \R$ via

\[ f_p(r) = \sum_{i = 0}^{2l}\dpn(g_r^i(p),g_r^{i+1}(p)). \]

Note that $f_p$ is continuous since $g_r$ is continuous as a function of $r$, and since $\dpn$ is continuous. Also note that $f_p(0) = 0$. We first consider the case that $n$ is even and $n>4l+2$. For any vertex $v \in P_n$ and $u = g_{r_n}^{-1}(v)$, the arc $[u,v]_{P_n}$ consists of exactly $\frac{n}{2}-1$ edges of $P_n$. So, we get $\dpn(u,v) = \frac{1}{n}(\frac{n}{2}-1) = \frac{1}{2}-\frac{1}{n}$. By Lemma~\ref{lem:min-winding}, we have $\dpn(u,v) \le \dpn(p,g_{r_n}(p))$, which provides the bound
\begin{align*}
	f_p(r_n) &= \sum_{i = 0}^{2l}\dpn(g_{r_n}^i(p),g_{r_n}^{i+1}(p)) \\
	&\ge \sum_{i = 0}^{2l}\bigg(\frac{1}{2}-\frac{1}{n}\bigg) = (2l+1)\bigg(\frac{1}{2}-\frac{1}{n}\bigg) = l + \frac{1}{2} - \frac{2l+1}{n} \\
	&>l\quad\mbox{ if }n>4l+2.
\end{align*}
We have shown that $f_p(r_n) > l$ in all even cases except $n = 4\ell+2$, and therefore the intermediate value theorem guarantees a solution $r\in[0,r_n]$ to the equation $f_p(r) = l$. In turn, this guarantees the existence of the desired star.

If $n$ is odd, then Lemma~\ref{lem:min-winding} again gives $\dpn(a,b) \le \dpn(p,g_{r_n}(p))$ for $v$ any vertex of $P_n$ and $u = g_{r_n}^{-1}(v)$, but now we resort to bounding $\dpn(u,v)$ instead of computing it directly. The arc $[u,v]_{P_n}$ contains $\frac{n+1}{2}-2$ complete edges of $P_n$, and part of one additional edge. To bound the proportion of this edge that is covered by $[u,v]_{P_n}$, let us realize $P_n$ geometrically in the plane, and assume without loss of generality that the line connecting $u$ and $v$ is given by $y = \sin(\frac{2\pi}{n}(\frac{n}{2}-\frac{1}{2})) = \sin(\frac{\pi}{n})$; see Figure~\ref{fig:min-winding}(right). It then follows that the proportion of the edge left uncovered is equal to $\sin(\frac{\pi}{n})/\sin(\frac{2\pi}{n}) = \frac{1}{2}\sec(\frac{\pi}{n})$. By Lemma~\ref{lem:technical-inequality}, we conclude that the proportion of this edge that is covered by $[u,v]_{P_n}$ is $1 - \frac{1}{2}\sec(\frac{\pi}{n}) > 1-\frac{1}{2}(\frac{n+1}{n-1}) = \frac{3}{2} - \frac{n}{n-1}$. Therefore, we have derived the bound $\dpn(u,v) > \frac{1}{n}(\frac{n+1}{2} - 2 + \frac{3}{2} - \frac{n}{n-1}) = \frac{1}{2} - \frac{1}{n-1}$, which lets us compute
\begin{align*}
	f_p(r_n) &= \sum_{i = 0}^{2l}\dpn(g_{r_n}^i(p),g_{r_n}^{i+1}(p)) \\
	&> \sum_{i = 0}^{2l}\bigg(\frac{1}{2}-\frac{1}{n-1}\bigg) = (2l+1)\bigg(\frac{1}{2}-\frac{1}{n-1}\bigg) = l + \frac{1}{2} - \frac{2l+1}{n-1} \\
	&\ge l\quad\mbox{ since }n \ge 4l+3.
\end{align*}
Hence $f_p(r_n) > l$, so once more the intermediate value theorem guarantees a solution $r\in[0,r_n]$ to $f_p(r) = l$, which proves the existence of the desired star.

The only outstanding case is $n = 4l+2$. This case is easily verified, as we can define the coordinates of the desired star explicitly: For $t \in [0,1]$, the point $p = (1-t)+\omega t$ is an arbitrary basepoint, and we construct the path of vertices $\{p, p\omega^{2l},p\omega^{4l},\dots p\omega^{(2l)2l}\}$. Observe that all of the adjacent distances in this path are equal to $r = \|p\omega^{2l j}-p\omega^{2l(j+1)}\| = \norm{1-\omega^{2l}}\cdot\norm{1-t+t\omega}$. As this value does not depend on $j$, the proposed path indeed corresponds to an inscribed equilateral $(2l+1)$-pointed star. But we must also check that this side length satisfies $r \leq r_n$. Observe that $r$ is a product of two factors, and expanding trigonometrically gives that the first is $\norm{1-\omega^{2l}} = 2\sin(4\pi l/n)$ and the second factor is $\norm{1-t+t\omega}^2 = 4\sin^2(\pi/n)t^2 - 4\sin^2(\pi/n)t + 1$. This function is quadratic in $t$, so we can easily check that its maximum value is 1 and that this occurs at $t = 0$ and $t = 1$. Hence, we have derived the bound
\[r \le 2\sin(4\pi l/n) = 2\sin(\pi(n-2)/n) = 2\sin(2\pi/n) \quad\mbox{since}\quad 4l = n - 2.\] 
Finally, Lemma~\ref{lem:sin-cos-inequality} gives that this is bounded above by $2\cos(\pi/n) = r_n$, as desired. This completes the proof.
\end{proof}

This is mind, we recall Definition~\ref{def:s-fns}. Given any positive integer $n \ge 4l+2$ and any point $p \in P_n$, we let $S_{2l+1}(p)$ denote the unique $(2l+1)$-pointed star inscribed in $P_n$ containing the vertex $p$. Furthermore, we let $s_{2l+1}(p)$ be the side length of this star (Figure~\ref{fig:P9_s_t}). For the remainder of this section, we will assume $n \ge 4l+2$. By Lemma~\ref{lem:s-cts}, $s_{2l+1}$ is a continuous function on a compact domain $P_n$, so we know it must achieve its extremal values. We thus turn to the task of finding these maxima and minima.

\begin{definition}
Say that a point $x \in P_n$ is a \emph{vertex crossing} if $S_{2l+1}(x)$ contains a vertex of $P_n$. Likewise, say that $x \in P_n$ is a \emph{midpoint crossing} if $S_{2l+1}(x)$ contains the midpoint of some edge of $P_n$. In either case, we call the point $x$ a \emph{crossing}. Moreover, with respect to a fixed star $S$ inscribed in $P_n$, we call a vertex of $S$ a \emph{vertex coincidence} if it also a vertex of $P_n$, or a \emph{midpoint coincidence} if it is also the midpoint of an edge of $P_n$. In either case, we call such a point a \emph{coincidence}. For a given a star, its \emph{coincidence number} is the total number of coincidences in its vertices.
\end{definition}

\begin{remark}\label{rem:expl-sl}
When $2l+1$ divides $n$, the proof of Lemma~\ref{lem:existence} shows that the barycentric coordinate of the points in an inscribed equilateral star are the same. This immediately implies that the only vertex crossings are vertices, that the only midpoint crossings are midpoints, and that we have an analytical formula for the side length function in terms for $t$: For $p(t) = (1-t)\omega +t\omega$ an arbitrary basepoint, we have
\[ s_{2l+1}(p(t)) = 2\sin\left(\frac{\pi l}{2l+1}\right)\sqrt{4\sin^2\left(\frac{\pi}{n}\right)t^2-4\sin^2\left(\frac{\pi}{n}\right)t+1}.\]

Recall from Definition~\ref{def:s,t} that $s_{P_n,l}$ and $t_{P_n,l}$ (henceforth denoted by $s_{n,l}$ and $t_{n,l}$ for brevity) are respectively equal to the global minimum and global maximum of $s_{2l+1}$ on $P_n$. In the case that $2l+1$ divides $n$, we can therefore derive an analytical formulae for these values. The radicand of the side length function is quadratic in $t$, so its minimum occurs at $t=1/2$ and gives $s_{n,l} = 2\sin(\frac{\pi l}{2l+1})\cos(\frac{\pi}{n})$.
Likewise, its maximum\footnote{Let $V=S^1$ be the circle of unit radius, equipped with the Euclidean metric. It follows from~\cite{AA-VRS1} (after modifying the metric) that we have $s_{V,l}=t_{V,l}=2\sin(\frac{\pi l}{2l+1})$ for all $l\ge 0$. 
In the case of the regular polygons $P_n$, note that if $(2l+1)\mid n$, then we have $t_{n,l}=t_{V,l}$ for all $l\ge 0$. Furthermore, if we let $n\to\infty$, while restricting attention to those $n$ with $(2l+1)\mid n$, then we get that $s_{n,l}\to s_{V,l}$ for all $l\ge 0$. This makes sense since $P_n$ converges to $S^1$ (for example in the Hausdorff distance) as $n\to\infty$.} occurs at $t=0$ and $t=1$ and gives $t_{n,l} = 2\sin(\frac{\pi l}{2l+1})$. 
\end{remark}

The following result shows that vertex coincidences and midpoint coincidences are disjointly supported.

\begin{lemma}\label{lem:vertex-midpt-impossible}
A star inscribed in $P_n$ cannot have both a vertex coincidence and a midpoint coincidence. Equivalently, a point in $P_n$ cannot be both a vertex crossing and a midpoint crossing.
\end{lemma}

\begin{proof}
Suppose that a $(2l+1)$-pointed star $S$ of side length $r$ is inscribed in $P_n$ and assume that $S$ contains both a vertex coincidence $v$ and a midpoint coincidence $m$, so that there is some integer $k$ that gives $v = g_r^k(m)$. Then consider the orbit $\{m,g_r^{k}(m),g_r^{2k}(m),\dots\}$ within $S$. It is clear that this orbit alternates between vertex coincidences and midpoint coincidences. Also, its size must divide $2l+1$, and hence is odd. This forces some element to be both a vertex coincidence and a midpoint coincidence, which is impossible. 
\end{proof}

\begin{lemma}
\label{lem:crossing-extrema}
Every crossing $p \in P_n$ is a local extrema of $\sl$. Moreover, all vertex crossings achieve the same value of $s_{2l+1}$, and all midpoint crossings achieve the same value of $s_{2l+1}$.
\end{lemma}

\begin{proof}
Let $p \in P_n$ be a crossing and let $q \in P_n$ be an arbitrary point whose Euclidean distance from $p$ is less than $r_n$. Then, $\norm{p-q}$ is monotonic (either increasing or decreasing) as $q$ approaches $p$ from the counterclockwise direction. However, the symmetry of the point $p$ tells us that $\norm{p-q}$ must take on identical values as $q$ approaches $p$ from the counterclockwise direction. Hence $p$ is a local extrema. 

The second part follows by symmetry of $P_n$, and by the existence and uniqueness in Lemmas~\ref{lem:existence} and \ref{lem:star-unique}.
\end{proof}

We have now proven a partial characterization of the extrema of the side length function. The following conjecture is much stronger.

\begin{conjecture}\label{conj:monotonic}
For $n \ge 4l+2$, every midpoint crossing of $P_n$ is a global minimum of $s_{2l+1}$, every vertex crossing is a global maximum of $s_{2l+1}$, the midpoint and vertex crossings are interleaved around $P_n$ (the counterclockwise traversal of all crossings in $P_n$ must alternate between vertex crossings and midpoint crossings), and $s_{2l+1}$ is strictly monotonic between adjacent midpoint and vertex crossings.
\end{conjecture}

We prove two special cases of this conjecture.

\begin{lemma}\label{lem:monotonic-divide}
Conjecture~\ref{conj:monotonic} is true if $(2l+1)\mid n$.
\end{lemma}

\begin{proof}
When $(2l+1)\mid n$, remark~\ref{rem:expl-sl} gives an explicit formula for $\sl$ as a function of the barycentric coordinate $t$, so we need only show that it is monotonically increasing on $t \in [0,\frac{1}{2}]$ and monotonically decreasing on $t \in [\frac{1}{2},1]$. This is easy: the radicand is of the form $ct^2-ct+1$ for $c > 0$, so we conclude that its only interior extremum is at $t = \frac{1}{2}$. Since $s_{2l+1}$ is monotonic between adjacent midpoint and vertex crossings, it follows from symmetry that all midpoint crossings are global minima of $s_{2l+1}$, and all vertex crossings are global maxima of $s_{2l+1}$.
\end{proof}

\begin{lemma}\label{lem:monotonic-1}
Conjecture~\ref{conj:monotonic} is true if $l=1$.
\end{lemma}

See Appendix~\ref{sec:monotonic-1} for the proof of Lemma~\ref{lem:monotonic-1}.

\subsection{Counting extrema of the side length function}\label{ss:counting}

We know that the number of periodic orbits in $\vrleq{P_n}{r}$ is equal to the number of stars of side length $r$ that can be inscribed in $P_n$, and this is closely related to the number of solutions to the equation $s_{2l+1}(p) = r$ for $p\in P_n$. The results of this subsection provide a method for counting the number of these solutions.

\begin{lemma}\label{lem:coincidence-num}
Let $n\ge 4l+2$. Any $(2l+1)$-pointed star $S$ inscribed in $P_n$ has coincidence number equal to either 0 or $\gcd(n,2l+1)$.
\end{lemma}

\begin{proof}
By Lemma~\ref{lem:vertex-midpt-impossible}, we can proceed by assuming that all coincidences are vertex coincidences; the proof is identical if we assume that all are midpoint coincidences. Suppose that a $(2l+1)$-pointed star $S$ is inscribed into $P_n$, and let $C$ be the set of coincidences of $S$ in $P_n$, denoting $d = |C|$. If $d = 0$, we are done. Otherwise, fix some $v \in C$ and define $m = \min\{i > 0: g_r^i(v)\in C\}$; note that this implies $\{v,g_r^m(v),g_r^{2m}(v),\dots\}\subseteq C$. We now prove that this set is indeed all of $C$. Towards a contradiction, suppose that there were some coincidence not contained in this orbit, say $g_r^{m'}(v)$. Then we have both $g_r^{am}(v)$ and $g_r^{bm'}(v)$ in $C$, so $g_r^{am+bm'}(v) \in C$ holds for all integers $a, b$. This proves that we have $g_r^e(v) \in C$ for $e = \gcd(m,m')$ which contradicts the minimality of $m$ if $e < m$. Hence, we have $e =m$, so $m$ divides $m'$, so $\{v,g_r^m(v),g_r^{2m}(v),\dots\} = C$. Note that in particular, we have $md = 2l+1$.

Now define another dynamical system $h: P_n \to P_n$ by $h(p) = \omega p$. Let $k = \min\{i > 0: h^i(v)\in C\}$, and note $\{v,h^{k}(v),h^{2k}(v),\dots\}\subseteq C$. If we now had some $h^{k'}(v)$ in $C$ but not in this orbit, then we would have $h^{ak+bk'}(v) \in C$ for all integers $a, b$, i.e.\ $h^e(v) \in C$ for $e = \gcd(k,k')$. As before, this is a contradiction unless $e = k$, so we establish $kd = n$. Hence, we see that $d$ is a common divisor of $n$ and $2l+1$, which gives $|C| \le \gcd(n,2l+1)$.

For the reverse inequality, it suffices to prove that if $d\mid n$, if $d\mid(2l+1)$ and if $v_0$ is a coincidence, then $g_r^{(2 l+1)/d}(v_0)$ is also a coincidence. This result will imply that $\{v_0,g_r^{(2 l+1)/d}(v_0),g_r^{2(2 l+1)/d}(v_0),\dots\} \subseteq C$ and hence that $|C| \ge \gcd(n,2l+1)$.

To see this, note that since $d\mid n$, we have that $\frac{l}{d}=\frac{l n}{d}\cdot\frac{1}{n}$ is an integer multiple of $\frac{1}{n}$. It follows that the point on $P_n$ at a counterclockwise distance of $\frac{l}{d}$ from vertex $v_0$ is another vertex of $P_n$. Let $v_1$ be the vertex of $P_n$ such that $\dpn(v_0,v_1) = \frac{l}{d} (\mbox{mod } 1)$, and similarly let $v_j$ be the vertex such that $\dpn(v_{0},v_j) = \frac{jl}{d} (\mbox{mod } 1)$ for each $j \in \{2,3, \ldots, d\}$. By the symmetry of $P_n$, we have that the value
\begin{equation}\label{eq:v1-v2}
\sum_{i=0}^{(2l+1)/d-1}\dpn(g_r^{i}(v_j),g_r^{i+1}(v_j))
\end{equation}
is fixed across all $j \in \{0,1, \ldots, d-1\}$. We now claim that this common value is $\frac{l}{d}$. To see this, first suppose towards a contradiction that it were strictly less than $\frac{l}{d}$. Then by cyclicity, we have $\sum_{i=0}^{2l}\dpn(g_r^{i}(v_0),g_r^{i+1}(v_0)) < \frac{dl}{d} = l$. But this is a contradiction, because $v_0$ is taken to be a vertex of an inscribed $(2l+1)$-star of radius $r$.  Likewise, if (\ref{eq:v1-v2}) were strictly greater than $\frac{l}{d}$, then this would give $\sum_{i=0}^{2l}\dpn(g_r^{i}(v_0),g_r^{i+1}(v_0)) > \frac{dl}{d} = l$, which is again a contradiction. Thus we have that the common value of (\ref{eq:v1-v2}) is exactly $\frac{l}{d}$. But because $v_j$ is defined to the unique point on $P_n$ whose distance from $v_0$ is $\frac{jl}{d} (\mbox{mod } 1)$, it follows that $g_r^{j(2l+1)/d}(v_0) = v_j$, i.e.\ $g_r^{j(2l+1)/d}(v_0)$ is a coincidence. Combining these, we have proven exactly $|C| = \gcd(n,2l+1)$.
\end{proof}

\begin{lemma}\label{lem:num-crossings}
Let $n\ge 4l+2$. The number of vertex crossings in $P_n$ and the number of midpoint crossings in $P_n$ are each equal to $n(2l+1)/\gcd(n,2l+1) = \lcm(n,2l+1)$.
\end{lemma}

\begin{proof}
The proof is identical for both vertex crossings and midpoint crossings. Let $p \in P_n$ be an arbitrary point, and consider how the star $S_{2l+1}(p)$ changes as $p$ winds once around $P_n$, in, say, the counterclockwise direction. Observe that $p$ is a vertex coincidence at $n$ different points, and, by symmetry, the same happens for all $2l+1$ vertices of $S_{2l+1}(p)$. Hence, $n(2l+1)$ vertex coincidences occur. However, by Lemma~\ref{lem:coincidence-num}, all coincidences occur in sets of size gcd$(n,2l+1)$. Therefore, the total number vertex crossings is exactly $n(2l+1)/\gcd(n,2l+1) = \lcm(n,2l+1)$.
\end{proof}

\begin{figure}[h]
\centering
\begin{tikzpicture}
\newdimen\rad
\rad=2cm
\newdimen\wid
\wid=0.02cm
\definecolor{bg}{RGB}{230,230,230}
\definecolor{mincolor}{RGB}{25, 50, 196}
\definecolor{maxcolor}{RGB}{74,195,63}

\foreach \rot in {0,60,...,359} {
	\draw[rotate around={\rot:(0,0)},line width=\wid,maxcolor] (1*\rad,0*\rad) -- (-0.5*\rad,0.86603*\rad) -- (-0.5*\rad,-0.86603*\rad) --  cycle;
}

\foreach \rot in {0,60,...,359} {
	\draw[rotate around={\rot:(0,0)},line width=\wid,mincolor] (0.75*\rad,0.433013*\rad) -- (-0.75*\rad,0.433013*\rad) -- (0*\rad,-0.86603*\rad) --  cycle;
}

\draw[line width=0.06cm,black] (0:\rad) \foreach \x in {60,120,...,359} {
	-- (\x:\rad)
} -- cycle (90:\rad);
\end{tikzpicture}\qquad\begin{tikzpicture}

\newdimen\rad
\rad=2cm
\newdimen\wid
\wid=0.02cm
\definecolor{bg}{RGB}{230,230,230}
\definecolor{mincolor}{RGB}{25, 50, 196}
\definecolor{maxcolor}{RGB}{74,195,63}

\foreach \rot in {0,45,...,359} {
	\draw[rotate around={\rot:(0,0)},line width=\wid,maxcolor] (1*\rad,0*\rad) -- (-0.426244*\rad,0.823444*\rad) -- (-0.426244*\rad,-0.823444*\rad) --  cycle;
}

\foreach \rot in {0,45,...,359} {
	\draw[rotate around={\rot:(0,0)},line width=\wid,mincolor] (0.853553*\rad,0.353553*\rad) -- (-0.765100*\rad,0.567100*\rad) -- (-0.140997*\rad,-0.941597*\rad) --  cycle;
}

\draw[line width=0.06cm,black] (0:\rad) \foreach \x in {45,90,...,359} {
	-- (\x:\rad)
} -- cycle (90:\rad) ;
\end{tikzpicture}\qquad\begin{tikzpicture}

\newdimen\rad
\rad=2cm
\newdimen\wid
\wid=0.02cm
\definecolor{bg}{RGB}{230,230,230}
\definecolor{mincolor}{RGB}{25, 50, 196}
\definecolor{maxcolor}{RGB}{74,195,63}

\foreach \rot in {0,32.7272,...,359} {
	\draw[rotate around={\rot:(0,0)},line width=\wid,maxcolor] (1*\rad,0*\rad) -- (-0.753409*\rad,0.602405*\rad) -- (0.294555*\rad,-0.927009*\rad) -- (0.294611*\rad,0.927001*\rad) -- (-0.753440*\rad,-0.602358*\rad) -- cycle;
}

\foreach \rot in {0,32.7272,...,359} {
	\draw[rotate around={\rot:(0,0)},line width=\wid,mincolor] (0.920627*\rad,0.270320*\rad) -- (-0.930065*\rad,0.327523*\rad) -- (0.531924*\rad,-0.808676*\rad) -- (0.010224*\rad,0.967890*\rad) -- (-0.605195*\rad,-0.778431*\rad) -- cycle;
}

\draw[line width=0.06cm,black] (0:\rad) \foreach \x in {32.7272,65.4545,...,359} {
	-- (\x:\rad)
} -- cycle (90:\rad);
\end{tikzpicture}
\caption{Geometric realizations of $s_{n,l}$ and $t_{n,l}$ in blue and green, respectively, for $(n,l) = (6,1)$ (left), $(8,1)$ (center) and $(11,2)$ (right).}
\label{fig:P9_s_t}
\end{figure}
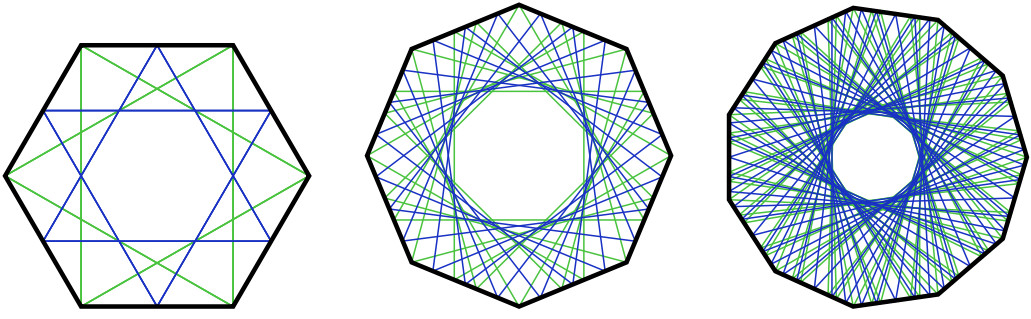

\begin{corollary}\label{cor:orbit-counting}
If Conjecture~\ref{conj:monotonic} is true,  then the number of equilateral $(2l+1)$-pointed stars of side length $r$ that can be inscribed into $P_n$ is equal to
\[ \begin{cases}
n/\gcd(n,2l+1) &\mbox{ if } r = s_{n,l} \mbox{ or } t_{n,l} \\
2n/\gcd(n,2l+1) &\mbox{ if } s_{n,l} < r < t_{n,l} \\
0 &\mbox{ otherwise}
\end{cases} \]
\end{corollary}

\begin{proof}
Recall that $s_{2l+1}$ is continuous on $P_n$, and that every $(2l+1)$-star contains $2l+1$ points. Hence, counting the number of inscribed stars of side length $r$ amounts to counting the number of points $p\in P_n$ for which $s_{2l+1}(p) = r$, and then dividing by $2l+1$. By Conjecture~\ref{conj:monotonic}, $s_{2l+1}$ oscillates between its global minimum and global maximum. Hence, the number of points $p$ for which $s_{2l+1}(p) = s_{n,l}$ or $s_{2l+1}(p) = t_{n,l}$ are respectively equal to the number of midpoint crossings and vertex crossings in $P_n$. By Lemma~\ref{lem:num-crossings}, this is exactly $n(2l+1)/\gcd(n,2l+1)$, which implies the existence of $n/\gcd(n,2l+1)$ stars of this side length. For $r \in (s_{n,l},t_{n,l})$, Conjecture~\ref{conj:monotonic} guarantees that the number of intersection points is twice the number of global minima or maxima. That is, there are $2n(2l+1)/\gcd(n,2l+1)$ intersections points and therefore $2n/\gcd(n,2l+1)$ stars. Finally, it is clear from the definition of $s_{n,l}$ and $t_{n,l}$ that there are no stars with side length $r \notin [s_{n,l},t_{n,l}]$. This completes the proof.
\end{proof}

\section{Vietoris--Rips complexes of regular polygons}\label{sec:vr}

As our main theorem, we describe the homotopy type of $\vrc{P_n}{r}$ when $r<r_n$, i.e.\ when $\vr{P_n}{r}$ is a cyclic graph. The following result assumes Conjecture~\ref{conj:monotonic}. However, we know by virtue of Lemmas~\ref{lem:monotonic-divide} and \ref{lem:monotonic-1} that the result holds when $(2l+1)\mid n$ or when $l = 1$. For ease of notation, we let $\qnl=\frac{n}{\gcd(n,2l+1)}$, and we would like to emphasize that $t_{n,0}=0$.


\begin{theorem}\label{thm:main-general}
Suppose $n\ge4l+2$ is such that Conjecture~\ref{conj:monotonic} is true. Then for $r<r_n$, we have\footnote{We can also describe the homotopy type of $\vrless{P_n}{r_n}$, or $\vrleq{P_n}{r_n}$ when $n$ is even (see Remark~\ref{rem:rn}), but for the sake of simplicity we omit this here.}
\begin{align*}
\vrcless{P_n}{r}&\simeq\begin{cases}
\bigvee^{\qnl-1}S^{2l} &\mbox{when }s_{n,l}<r\le t_{n,l}\\
S^{2l+1} & \mbox{when } t_{n,l}<r\le s_{n,l+1}
\end{cases}\\
\vrcleq{P_n}{r}&\simeq\begin{cases}
\bigvee^{\qnl-1}S^{2l} &\mbox{when }r=s_{n,l}\\
\bigvee^{3\qnl-1} S^{2l} &\mbox{when }s_{n,l}<r<t_{n,l}\\
\bigvee^{2\qnl-1}S^{2l} &\mbox{when }r=t_{n,l}\\
S^{2l+1} &\mbox{when } t_{n,l}<r< s_{n,l+1}.
\end{cases}
\end{align*}
Furthermore, 
\begin{itemize}
\item For $s_{n,l}<r<\tr\le t_{n,l}$ or $t_{n,l}<r<\tr\le s_{n,l+1}$, the inclusion $\vrcless{P_n}{r}\hookrightarrow\vrcless{P_n}{\tr}$ is a homotopy equivalence.
\item For $t_{n,l}<r<\tr< s_{n,l+1}$, inclusion $\vrcleq{P_n}{r}\hookrightarrow\vrcleq{P_n}{\tr}$ is a homotopy equivalence.
\item For $s_{n,l}\le r<\tr\le t_{n,l}$, inclusion $\vrcleq{P_n}{r}\hookrightarrow\vrcleq{P_n}{\tr}$ induces a rank $\qnl-1$ map on $2l$-dimensional homology $H_{2l}(-;\F)$ for any field $\F$.
\end{itemize}
\end{theorem}

\begin{proof}
Let $\{v_0,\ldots,v_{n-1}\}$ be the set of vertices of $P_n$ (i.e., the set of $n$ evenly-spaced points on the circle defining $P_n$), and let $\{m_0,\ldots,m_{n-1}\}$ be the set of midpoints of edges of $P_n$. Furthermore, choose both sets to be cyclically ordered, giving
\[ v_0\prec m_0\prec v_1\prec m_1\prec\ldots\prec m_{n-2}\prec v_{n-1}\prec m_{n-1}\prec v_0 .\]
We will consider first the homotopy types of $\vrcless{P_n}{r}$, then the homotopy types of $\vrcleq{P_n}{r}$, then the inclusion maps for $r<\tr$ in the $<$ case, and finally the inclusion maps in the $\le$ case.

For the homotopy types of $\vrcless{P_n}{r}$, if $t_{n,l}<r\le s_{n,l+1}$, then $\frac{l}{2l+1}<\wf(\vrless{P_n}{r})<\frac{l+1}{2l+3}$ by Definition~\ref{def:s,t}, and so $\vrcless{P_n}{r}\simeq S^{2l+1}$ by Theorem~\ref{thm:homotopy-generic}. Alternatively, if $s_{n,l}<r\le t_{n,l}$, then $\wf(\vrless{P_n}{r})=\frac{l}{2l+1}$. 
By Conjecture~\ref{conj:monotonic} (which we know in the cases of Lemma~\ref{lem:monotonic-divide} and Lemma~\ref{lem:monotonic-1}) and the intermediate value theorem, for all $0\le j\le n$ there are points
\[ v_j \prec a_j \prec m_j \prec b_j \prec v_{j+1}\prec v_j\]
such that we have $\qnl$ invariant sets of fast points
\[ I_j(r) = (a_j,b_j)_{P_n}\cup(a_{j+\qnl},b_{j+\qnl})_{P_n}\cup\ldots\cup(a_{j+2l\qnl},b_{j+2l\qnl})_{P_n}\]
and $\qnl$ invariant sets of slow points
\[[b_j,a_{j+1}]_{P_n}\cup[b_{j+\qnl},a_{j+\qnl+1}]_{P_n}\cup\ldots\cup[b_{j+2l\qnl},a_{j+2l\qnl+1}]_{P_n}.\]
To see how the count is obtained, recall from Lemma~\ref{lem:num-crossings} that the number of vertex crossings and midpoint crossings in $P_n$ are each equal to $\qnl(2l+1)$. Each invariant set of fast (resp.\ slow) points is a union of $(2l+1)$ segments, and hence there are $\qnl$ invariant sets of fast (resp.\ slow) points. It follows from the existence of fast points that the supremum in the definition of the winding fraction is attained (Definition~\ref{defn:inf-cyc}). Moreover, $\vrless{P_n}{r}$ has no periodic orbits and $\qnl$ permanently fast orbits. By Theorem~\ref{thm:homotopy-generic} we have $\vrcless{P_n}{r}\simeq \bigvee^{\qnl-1}S^{2l}$.

Next consider $\vrcleq{P_n}{r}$. If $t_{n,l}<r<s_{n,l+1}$, then $\frac{l}{2l+1}<\wf(\vrleq{P_n}{r})<\frac{l+1}{2l+3}$, and so $\vrcleq{P_n}{r}\simeq S^{2l+1}$ by Theorem~\ref{thm:homotopy-generic}.

When $r=s_{n,l}$, the cyclic graph $\vrleq{P_n}{s_{n,l}}$ has winding fraction $\frac{l}{2l+1}$, no fast points, and $\qnl$ periodic orbits given by the midpoints of edges in $P_n$ (by Corollary~\ref{cor:orbit-counting}, assuming Conjecture~\ref{conj:monotonic}). It follows from Theorem~\ref{thm:homotopy-generic} that $\vrleq{P_n}{s_{n,l}}\simeq\bigvee^{\qnl-1}S^{2l}$.

If $s_{n,l}<r< t_{n,l}$, then $\wf(\vrleq{P_n}{r})=\frac{l}{2l+1}$. By Conjecture~\ref{conj:monotonic}, Corollary~\ref{cor:orbit-counting}, and the intermediate value theorem, there are points
\[ v_j \prec a_j \prec m_j \prec b_j \prec v_{j+1}\prec v_j\]
for all $0\le j\le n$ such that we have that we have
\begin{itemize}
\item $\qnl$ periodic orbits of the form $A_j(r)=\{a_j, a_{j+\qnl}, a_{j+2\qnl}, \ldots, a_{j+2l\qnl}\}$,
\item $\qnl$ periodic orbits of the form $B_j(r)=\{b_j, b_{j+\qnl}, b_{j+2\qnl}, \ldots, b_{j+2l\qnl}\}$,
\item $\qnl$ invariant sets of fast points
\[ I_j(r) = (a_j,b_j)_{P_n}\cup(a_{j+\qnl},b_{j+\qnl})_{P_n}\cup\ldots\cup(a_{j+2l\qnl},b_{j+2l\qnl})_{P_n},\]
\item and $\qnl$ invariant sets of slow points
\[(b_j,a_{j+1})_{P_n}\cup(b_{j+\qnl},a_{j+\qnl+1})_{P_n}\cup\ldots\cup(b_{j+2l\qnl},a_{j+2l\qnl+1})_{P_n}.\]
\end{itemize}
The supremum in the definition of the winding fraction is attained, and $\vrleq{P_n}{r}$ has $p=2\qnl$ periodic orbits and $F=\qnl$ permanently fast orbits. By Theorem~\ref{thm:homotopy-generic} we have $\vrc{P_n}{r}\simeq \bigvee^{3\qnl-1}S^{2l}$.

When $r=t_{n,l}$, the cyclic graph $\vrleq{P_n}{t_{n,l}}$ has winding fraction $\frac{l}{2l+1}$. There are $\qnl$ periodic orbits given by the vertices of $P_n$. The remaining points of $P_n$ are divided into $\qnl$ invariant sets of permanently fast points. Hence by Theorem~\ref{thm:homotopy-generic} we have $\vrleq{P_n}{t_{n,l}}\simeq\bigvee^{2\qnl-1}S^{2l}$.

In order to study the inclusion maps for $r<\tr$, we will need some more notation from~\cite{AAR}. Given a finite cyclic graph $G$ with dynamics $f\colon V\to V$ and $\wf(G)=\frac{\omega}{\ell}$, a finite cyclic graph $\tG$ with dynamics $\tf\colon\tG\to\tG$ and $\wf(\tG)=\frac{\omega}{\ell}$, and a cyclic graph homomorphism $h\colon G\to\tG$, we say that a periodic orbit of $\tG$ with vertex set $\{\tv_1,\ldots,\tv_\ell\}$ is \emph{hit} by $h$ if there is a periodic orbit of $G$ with vertex set $\{v_1,\ldots,v_\ell\}$ such that $\tf^i(h(\{v_1,\ldots,v_\ell\}))=\{\tv_1,\ldots,\tv_\ell\}$ for some $i$ (\cite[Definition~4.2]{AAR}). We let $\hit(h)$ denote the number of distinct periodic orbits of $\tG$ hit by $h$.

The following is from~\cite[Section~5]{AAR}. For a graph $G$ with vertex set $V$, let $\fin(G)$ be the poset of all finite subsets of $V$, ordered by inclusion. For $G$ a closed and continuous cyclic graph that is singular, let $V_P\subseteq V$ be the set of periodic points. Let $\tfin(G)$ be the set of all finite $W\subseteq V$ such that
\begin{itemize}
\item $V_P\subseteq W$,
\item if $v\in W$ is a fast point of $G$ that achieves periodicity (with $f_m(v)$ periodic), then for all $1\le i\le m$ the point $f_i(v)$ is in $W$.
\item for every invariant set of permanently fast points $I$ of $G$, the induced graph $G[I\cap W]$ on the intersection contains a single periodic orbit isomorphic to $C_{2l+1}^l$.
\end{itemize}
Lemma~5.9 of~\cite{AAR} shows that for $G$ closed, continuous, singular, and with a finite number of periodic points and invariant sets of permanently fast points, the poset $\tfin(G)$ is cofinal in $\fin(G)$. In other words, given any set $W\in\fin(G)$ we can find a potentially larger set $\tW$ containing $W$ and satisfying $\tW\in\tfin(G)$

We now consider the inclusion maps in the $<$ case. For $s_{n,l}<r<\tr\le t_{n,l}$, it follows from Theorem~\ref{thm:homotopy-generic} that the inclusion $\vrcless{P_n}{r}\hookrightarrow\vrcless{P_n}{\tr}$ is a homotopy equivalence.

To prove that the inclusion $\vrcless{P_n}{r}\hookrightarrow\vrcless{P_n}{\tr}$ is a homotopy equivalence for $t_{n,l}<r<\tr\le s_{n,l+1}$, choose a finite subset $W\in\tfin(\vrless{P_n}{r})$. Using the cofinality in~\cite[Lemma~5.9]{AAR}, we can also find a set $\tW$ satisfying $W\subseteq\tW\in\tfin(\vrless{P_n}{\tr})$. The following commutative diagram is given by inclusions.
\[ \xymatrix{
\vrcless{W}{r} \ar@{->}[r] \ar@{->}[d]^{\iota_W} & \vrcless{P_n}{r} \ar@{->}[d]^{\iota_{P_n}}\\
\vrcless{\tW}{\tr} \ar@{->}[r] & \vrcless{P_n}{\tr}
} \]
The proof of~\cite[Theorem~5.3]{AAR} gives that the horizontal maps are homotopy equivalences. Note that $\hit(\iota_W)=q$ since for all $0\le j<q$, the periodic orbit for $\vrless{W}{r}$ in $I_j(r)$ hits the periodic orbit for $\vrless{\tW}{\tr}$ in $I_j(\tr)$. The map $\iota_W$ is a homotopy equivalence by~\cite[Proposition~4.2]{AAR}, and therefore $\iota_{P_n}$ is a homotopy equivalence as well.

For the $\le$ case, Theorem~\ref{thm:homotopy-generic} implies that $\vrcleq{P_n}{r}\hookrightarrow\vrcleq{P_n}{\tr}$ is a homotopy equivalence for $t_{n,l}<r<\tr< s_{n,l+1}$.

Suppose now that $s_{n,l}<r<\tr< t_{n,l}$; we will show that the inclusion $\vrcleq{P_n}{r}\hookrightarrow\vrcleq{P_n}{\tr}$ induces a rank $\qnl-1$ map on $2l$-dimensional homology $H_{2l}(-;\F)$ for any field $F$. By cofinality, let $W\in\tfin(\vrleq{P_n}{r})$ and $\tW\in\tfin(\vrleq{P_n}{\tr})$ be finite subsets with $W\subseteq\tW$. Consider the following commutative diagrams.
\[ 
\xymatrix{
\vrcleq{W}{r} \ar@{->}[r]^{\iota_r} \ar@{->}[d]^{\iota_W} & \vrcleq{P_n}{r} \ar@{->}[d]^{\iota_{P_n}}\\
\vrcleq{\widetilde{W}}{\tilde{r}} \ar@{->}[r]^{\iota_{\tilde{r}}} & \vrcleq{P_n}{\tilde{r}}
}
\hspace{3mm} 
\xymatrix{
H_{2l}(\vrcleq{W}{r};\F) \ar@{->}[r]^{\iota_r^*} \ar@{->}[d]^{\iota_W^*} & H_{2l}(\vrcleq{P_n}{r};\F) \ar@{->}[d]^{\iota_{P_n}^*}\\
H_{2l}(\vrcleq{\widetilde{W}}{\tilde{r}};\F) \ar@{->}[r]^{\iota_{\tilde{r}}^*} & H_{2l}(\vrcleq{P_n}{\tilde{r}};\F)
}
\]
The proof of~\cite[Theorem~5.3]{AAR} gives that the horizontal maps $\iota_r$ and $\iota_{\tr}$ are homotopy equivalences, and hence the horizontal maps  $\iota_r^*$ and $\iota_{\tr}^*$ are isomorphisms. To see that $\hit(\iota_W)=\qnl$, first note that $\vrleq{W}{r}$ has $3\qnl$ periodic orbits: the $\qnl$ periodic orbits $A_j(r)$, the $\qnl$ periodic orbits $B_j(r)$, and the $\qnl$ periodic orbits in $I_j(r)$ for each $0\le j<\qnl$. We have the analogous periodic orbits in $\vrleq{\tW}{\tr}$, after replacing $r$ everywhere by $\tr$. The periodic orbits in $W$ corresponding to $A_j(r)$, $B_j(r)$, and $I_j(r)$ map under $\iota_W$ to the periodic orbit corresponding to $I_j(\tr)$ in $\tW$. It follows that $\hit(\iota_W)=\qnl$. By~\cite[Proposition~4.2]{AAR} we have $\rank(\iota_W^*)=\qnl-1$, and therefore $\rank(\iota_{P_n}^*)=\qnl-1$. An analogous proof works in the more general case when $s_{n,l}\le r<\tr\le t_{n,l}$.\end{proof}

As special cases, we obtain the following theorems which are true by virtue of Lemmas~\ref{lem:monotonic-divide} and \ref{lem:monotonic-1}. 

\begin{theorem}\label{thm:main-1}
Suppose $n\ge 6$ and $l=1$. Then for $r<r_n$, we have
\begin{align*}
\vrcless{P_n}{r}&\simeq\begin{cases}
S^1 & \mbox{when } 0<r\le s_{n,1}\\
\bigvee^{q_{n,1}-1}S^2 &\mbox{when }s_{n,1}<r\le t_{n,1}\\
S^3 & \mbox{when } t_{n,1}<r\le s_{n,2}
\end{cases}\\
\vrcleq{P_n}{r}&\simeq\begin{cases}
S^1 &\mbox{when } 0<r< s_{n,1}\\
\bigvee^{q_{n,1}-1}S^2 &\mbox{when }r=s_{n,1}\\
\bigvee^{3q_{n,1}-1}S^2 &\mbox{when }s_{n,1}<r<t_{n,1}\\
\bigvee^{2q_{n,1}-1}S^2 &\mbox{when }r=t_{n,1}\\
S^3 &\mbox{when } t_{n,1}<r< s_{n,2}
\end{cases}
\end{align*}
Furthermore,
\begin{itemize}
\item For $0<r<\tr\le s_{n,1}$, for $s_{n,1}<r<\tr\le t_{n,1}$, or for $t_{n,1}<r<\tr\le s_{n,2}$, the inclusion $\vrcless{P_n}{r}\hookrightarrow\vrcless{P_n}{\tr}$ is a homotopy equivalence.
\item For $0<r<\tr< s_{n,1}$ or $t_{n,1}<r<\tr< s_{n,2}$, the inclusion $\vrcleq{P_n}{r}\hookrightarrow\vrcleq{P_n}{\tr}$ is a homotopy equivalence.
\item For $s_{n,1}\le r<\tr\le t_{n,1}$, the inclusion $\vrcleq{P_n}{r}\hookrightarrow\vrcleq{P_n}{\tr}$ induces a map of rank $q_{n,1}-1$ on $2$-dimensional homology $H_2(-;\F)$ for any field $\F$.
\end{itemize}
\end{theorem}

\begin{theorem}\label{thm:main}
Suppose $n=q(2l+1)$ with $q\ge 2$. Then for $r < r_n$, we have
\begin{align*}
\vrcless{P_n}{r}&\simeq\begin{cases}
\bigvee^{q-1}S^{2l}&\mbox{when }s_{n,l}<r\le t_{n,l}\\
S^{2l+1}&\mbox{when }t_{n,l}<r\le s_{n,l+1}
\end{cases}\\
\vrcleq{P_n}{r}&\simeq\begin{cases}
\bigvee^{q-1}S^{2l}&\mbox{when }r=s_{n,l}\\ 
\bigvee^{3q-1}S^{2l}&\mbox{when }s_{n,l}<r< t_{n,l}\\
\bigvee^{2q-1}S^{2l}&\mbox{when }r=t_{n,l}\\
S^{2l+1}&\mbox{when }t_{n,l}<r< s_{n,l+1}.
\end{cases}
\end{align*}
Furthermore, 
\begin{itemize}
\item For $s_{n,l}<r<\tr\le t_{n,l}$ or $t_{n,l}<r<\tr\le s_{n,l+1}$, the inclusion $\vrcless{P_n}{r}\hookrightarrow\vrcless{P_n}{\tr}$ is a homotopy equivalence.
\item For $t_{n,l}<r<\tr< s_{n,l+1}$, inclusion $\vrcleq{P_n}{r}\hookrightarrow\vrcleq{P_n}{\tr}$ is a homotopy equivalence.
\item For $s_{n,l}\le r<\tr\le t_{n,l}$, inclusion $\vrcleq{P_n}{r}\hookrightarrow\vrcleq{P_n}{\tr}$ induces a rank $q-1$ map on $2l$-dimensional homology $H_{2l}(-;\F)$ for any field $\F$.
\end{itemize}
\end{theorem}

As a consequence, we can describe the persistent homology of $\vrc{P_n}{r}$ over the range of scale parameters $r\in(0,r_n)$, which we refer to as the \emph{restriction} of the persistent homology to $(0,r_n)$.

\begin{corollary}\label{cor:main}
Suppose $n=q(2l+1)$ is a multiple of $2l+1$. Then the restriction of the $(2l+1)$-dimensional persistent homology of $\vrcless{P_n}{r}$ (resp. $\vrcleq{P_n}{r}$) to $(0,r_n)$ consists of a single interval $(t_{n,l},s_{n,l+1}]$ (resp.\ $(t_{n,l},s_{n,l+1})$).

The restriction of the $2l$-dimensional persistent homology of $\vrcless{P_n}{r}$ to $(0,r_n)$ consists of $q-1$ intervals of the form $(s_{n,l},t_{n,l}]$. The restriction of the $2l$-dimensional persistent homology of $\vrcleq{P_n}{r}$ to $(0,r_n)$ consists of $q-1$ intervals of the form $[s_{n,l},t_{n,l}]$, as well as $2q$ points $[r,r]$ on the diagonal for every $s_{n,l}<r< t_{n,l}$, and $q$ points $[t_{n,l},t_{n,l}]$ on the diagonal.
\end{corollary}

\begin{corollary}\label{cor:main2}
When $n=(2k+1)!!$, then we can give a complete description of the persistent homology of $\vrc{P_n}{r}$ restricted to $(0,r_n)$.
\end{corollary}

As an example, see Figure~\ref{fig:P15PersistentHomotopyRips} for the persistent homology of $\vrc{P_{15}}{r}$ in all homological dimensions.



\section{Vietoris--Rips complexes of subsets of a polygon}\label{sec:subsets}

As a second main result (Theorem~\ref{thm:secondary}), we study finite subsets of the regular polygons. We show the counterintuitive result that the Vietoris-Rips complex of an arbitrarily dense finite sample from $P_n$ can have the homotopy type of a wedge sum of an essentially \emph{arbitrary} number of even spheres. More precisely, suppose $n = q(2l+1)$ with $l\ge 1$ and $q\ge 2$. Theorem~\ref{thm:secondary}(ii) states that for any $\varepsilon>0$, $z\ge q$, and $s_{n,l}<r<t_{n,l}$, there is an $\varepsilon$-dense finite subset $X\subseteq P_n$ such that $\vrc{X}{r} \simeq \bigvee^{z-1} S^{2l}$. 

In Question~\ref{ques:Latschev-higher} we ask: for $M$ a Riemannian manifold, is it true that $\vrcless{X}{r}\simeq\vrcless{M}{r}$ for $X$ a sufficiently dense depending on $r$? This is known to be true for $r$ small by Latschev's theorem~\cite{Latschev2001}, and also for all scale parameters $r$ in the case when $M=S^1$ is the circle~\cite{AA-VRS1}. To our knowledge, this question is unknown for a general Riemannian manifold. Our Theorem~\ref{thm:secondary} shows that Question~\ref{ques:Latschev-higher} has a negative answer when $M$ is not Riemannian, for example if $M=P_n$ is a regular polygon equipped with the Euclidean metric.

Two lessons to be learned from Theorem~\ref{thm:secondary} are that higher-dimensional homology can be ubiquitous in Vietoris--Rips complexes of the most simple (even planar) shapes, and that the choice of metric (for example Euclidean versus Riemannian) can have large effects on the homotopy type of the resulting Vietoris--Rips complexes. 

\begin{definition}
A subset $X\subseteq P_n$ is defined to be \emph{$\varepsilon$-dense} if for each point $p\in P_n$, there exists some point $x\in X$ with $\norm{x-p}\le\varepsilon$.
\end{definition}

Result (i) in the theorem below shows that for $s_{n,l}<r<t_{n,l}$, any sufficiently dense sample of the regular polygon $P_n$ produces a Vietoris--Rips complex homotopy equivalent to \emph{some} wedge of $2l$-spheres. Result (ii) shows that for essentially any $z \in \N$, we can \emph{construct} an arbitrarily dense finite sample of $P_n$ whose Vietoris-Rips complex is an $z$-fold wedge sum of even-dimensional spheres. These results are analogous to~\cite[Theorems~7.1,7.2]{AAR} in the case of the ellipse, except our results here are more general in that they hold not only for $l=1$ (corresponding to 2-dimensional spheres), but also for higher dimensions $l\ge 1$ (giving $2l$-dimensional spheres).

\begin{theorem}\label{thm:secondary}
Suppose $n = q(2l+1)$ with $l\ge 1$ and $q\ge 2$.
\begin{enumerate}
\item[(i)] For any sufficiently dense finite sample $X\subseteq P_n$ and $s_{n,l}<r<t_{n,l}$, we have $\vrc{X}{r}\simeq \bigvee^{z-1}S^{2l}$ for some $z \geq q$.
\item[(ii)] For any $\varepsilon>0$, $z\ge q$, and $s_{n,l}<r<t_{n,l}$, there is an $\varepsilon$-dense finite subset $X\subseteq P_n$ such that $\vrc{X}{r} \simeq \bigvee^{z-1} S^{2l}$.
\end{enumerate}
\end{theorem}

\begin{proof} [Proof of (i)] 
Throughout this proof, all the arithmetic operations on point indices will be taken modulo $n$. Also observe that we are back in the case of finite cyclic graphs, so the dynamical system $f$ (see Definition~\ref{defn:cyc-dyn-fin}) is a map $f\colon X \r X$.

As in the proof of Theorem~\ref{thm:main}, 
let $\{v_0,\ldots,v_{n-1}\}$ denote the vertices of $P_n$, 
and let $\{m_0,\ldots,m_{n-1}\}$ denote the midpoints of edges of $P_n$, ordered cyclically as
\[ v_0\prec m_0\prec v_1\prec m_1\prec\ldots\prec m_{n-2}\prec v_{n-1}\prec m_{n-1}\prec v_0 .\]
By Lemma~\ref{lem:monotonic-divide} and the intermediate value theorem, we get points
\[ v_j \prec a_j \prec m_j \prec b_j \prec v_{j+1}\prec v_j\]
for all $0\le j < n$, such that 
\begin{itemize}
\item $g_r(a_j)=a_{j+q}$ and $g_r(b_j)=b_{j+q}$ for all $0\le j<n$,
\item $\vr{P_n}{r}$ has $q$ invariant sets of permanently fast points
\[I^f_j:=(a_j,b_j)_{P_n}\cup(a_{j+q},b_{j+q})_{P_n}\cup\ldots\cup(a_{j+2lq},b_{j+2lq})_{P_n} \mbox{ for }0\le j<q\]
\item the points 
\[(b_j,a_{j+1})_{P_n}\cup(b_{j+q},a_{j+q+1})_{P_n}\cup\ldots\cup(b_{j+2lq},a_{j+2lq+1})_{P_n} \mbox{ for }0\le j<q\]
consist entirely of slow points of $\vr{P_n}{r}$.
\end{itemize}

Fix $0\leq j < n$.
Now take $X$ to be sufficiently dense so that there exist points $x \in (a_j,b_j)_{P_n}$ and $x' \in (b_j,a_{j+1})_{P_n}$ satisfying 
\[x \prec f^{2l+1}(x) \preceq f^{2l+1}(x') \prec x' \prec x.\]
Here the first $\prec$ relation holds (for $X$ sufficiently dense) because $x$ is a fast point of $\vr{P_n}{r}$, and the second $\prec$ relation holds because $x'$ is a slow point of $\vr{P_n}{r}$. 
We claim that there exists some point $z \in X$ such that $x \preceq z \preceq x' \prec x$ and $f^{2l+1}(z) = z$. Towards a contradiction, suppose this is not the case. 
Fix an enumeration $\{x_1,x_2,\ldots, x_k\}$ of all the points in $[x,x']_X$ so that 
\[x = x_1 \prec x_2 \prec \ldots \prec x_k = x' \prec x.\]
Now $x_1$ is a fast point of $f$ and $x_k$ is a slow point of $f$. Let $x_J$ be the slow point of $f$ in $ [x,x']_{X}$ with the smallest index in the counterclockwise order. Then $x_{J-1}$ is a fast point of $f$, and we have
\[ x_{J-1} \prec f^{2l+1}(x_{J-1}) \preceq f^{2l+1}(x_{J}) \prec x_J \prec x_{J-1}.\]
But we have already enumerated all the points in $[x,x']_X$, and the preceding line is a contradiction to $(x_{J-1},x_J)_X$ being empty. This proves the claim, giving a periodic point $z\in(a_j,a_{j+1})_{P_n}$ with $f^{2l+1}(z)=z$. 

Since $0\leq j < n$ was arbitrary, we obtain $n$ such periodic points, each having an orbit of length $2l+1$. Thus there are at least $q$ periodic orbits of length $2l+1$. An application of Proposition~\ref{prop:vrchtpy} now shows that $\vrc{X}{r}\simeq \bigvee^{z-1}S^{2l}$ for some $z \geq q$, proving (i). \qedhere
\end{proof}

We introduce some notation that will be useful for the proof of (ii). Recall from Definition~\ref{def:gv} that for any point $x \in P_n$, the point $g_{s_{2l+1}(x)}(x)$ is the vertex of the $(2l+1)$-star $S_{2l+1}(x)$ adjacent to $x$ in the counterclockwise direction. For convenience, we adopt the following notation: given $x_1 \in P_n$, write $x_1\prec x_2\prec \ldots\prec x_{2l+1}\prec x_1$ to denote the vertices of the star $S_{2l+1}(x)$. Note that by the symmetries of $P_n$ (with $n=q(2l+1)$), we have that $\|x_1-x'_1\|=\|x_i-x'_i\|$ for all $1\le i\le 2l+1$. 

\begin{figure}
\def\svgwidth{0.8\linewidth}
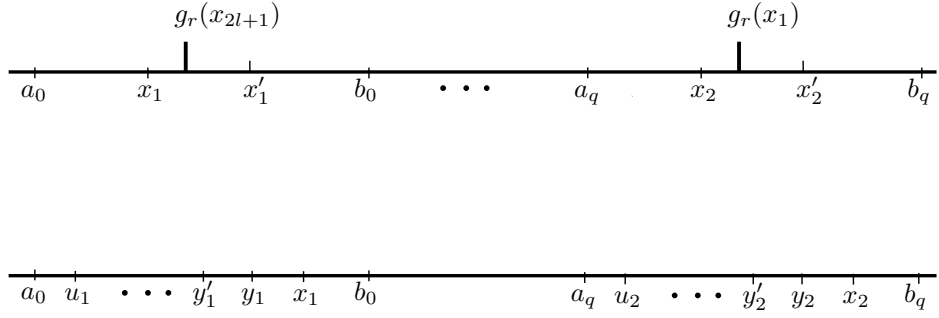
\caption{Figure for the proof of Theorem~\ref{thm:secondary}(ii)}
\label{fig:epsilon-density}
\end{figure}

\begin{proof}[Proof of (ii)]
The proof involves building up an $\varepsilon$-dense set $X$: the construction is illustrated in Figure~\ref{fig:epsilon-density}.
We use the setup of the proof of (i), with invariant sets of permanently fast points $I^f_j$ marked by $a_j,b_j$ terms, and with each $I^s_j$ consisting of slow points of $\vr{P_n}{r}$.
We construct $X$ in two steps: first we insert points in $X$ that will be periodic under $f\colon X\to X$, and then we append extra non-periodic points in order to to achieve $\varepsilon$-density. The periodic points will be arranged inside the regions $I^f_j$. Given $z\ge q$, we make an arbitrary choice of nonnegative integers $z_0,z_1,\ldots, z_{q-1}$ such that $\sum_{j=0}^{q-1} z_j = z$; we will construct $z_j$ periodic orbits inside each region $I^f_j$.

First let $x_1 \in (a_0,b_0)_{P_n}$ be such that $\norm{x_1 - b_0} < \varepsilon$, and add $\{x_1, x_2,\ldots, x_{2l+1}\}$ to $X$. Since $x_1$ is in an invariant set of permanently fast points for $\vr{P_n}{r}$, we know that for all $1\le j\le 2l+1$, we have $\gpn_r(x_j) \in (x_{j+1}, \gpn_r(b_j))_{P_n} = (x_{j+1}, b_{j+q})_{P_n}$, where in particular $\gpn_r(x_{2l+1}) \in (x_1,b_0)_{P_n}$.
We do not add to $X$ any points in the set
\[(x_2,\gpn_r(x_1)]_{P_n} \cup 
(x_3,\gpn_r(x_2)]_{P_n} \cup \ldots \cup
(x_1,\gpn_r(x_{2l+1})]_{P_n}. \]
Then we necessarily have $f(x_1) = x_2$, $f(x_2) = x_3$, \ldots, and $f(x_{2l+1}) = x_1$. This gives us our first periodic orbit in $X$. 

Next let $x'_1\in (\gpn_r(x_{2l+1}),b_0)_{P_n}$ be such that 
\[x'_2 \in (\gpn_r(x_1),b_q)_{P_n}, \,
x'_3 \in (\gpn_r(x_2),b_{2q})_{P_n}, \ldots, \,
x'_{2l+1} \in (\gpn_r(x_{2l}),b_{2lq})_{P_n}.\] 
Repeating the process above, we get another periodic orbit $\{x'_j\}_{j=1}^{2l+1}$. We iterate this process $z_1$ times to get $z_1$ periodic orbits inside $I^f_1$. 

The next step is to attain $\varepsilon$-density inside $I^f_1$. 
We first make a preliminary observation: by continuity of $g_r$ (Lemma~\ref{lem:g_r-cont-p}), for $z \in (a_0, x_1)_{P_n}$ such that $\norm{z - x_1}$ is small enough, we have $\gpn_r(z) \in (x_2, \gpn_r(x_1))_{P_n}$.

Now let $u_1 \in (a_0,b_0)_{P_n}$ be such that $\norm{a_0-u_1} < \varepsilon$. By a compactness argument, there exists a $0 < \delta < \frac{\epsilon}{2}$ such that for all $y_1 \in [u_1,x_1]_{P_n}$ and $\tilde{y} \in (a_0, y_1)_{P_n}$ with $\norm{y-\tilde{y}} < \delta$, we have $\gpn_r(\tilde{y}) \in (y_2, \gpn_r(y_{1}))_{P_n}$.

Now let $y_1 \in (a_0,x_1)_{P_n}$ be such that $\norm{y_1 - x_1}  = \delta$. Add $y_1,\ldots, y_{2l+1}$ to $X$. By the choices made above, we know $f(y_1) = x_2 \neq y_2$, and likewise for the points $f(y_2)$, $f(y_3)$, and so on. Thus adding $\{y_j\}_{j=1}^{2l+1}$ to $X$ does not add a periodic orbit. We iterate this construction until we achieve $\varepsilon$-density in the segments $[u_j,x_j]_{P_n}$ for $1\leq j \leq 2l+1$. This achieves $\varepsilon$-density in $I^f_1$.

We repeat the process of adding periodic orbits and then non-periodic orbits to achieve $\varepsilon$-density for all $\{I^f_j\}_{j=0}^{q-1}$. This yields $z = \sum_{j=0}^{q-1} z_j$ periodic orbits inside $X$. 

Finally, we add sufficiently many points from the slow regions $\{I^s_j\}_{j=0}^{q-1}$ to get $X$ to be $\varepsilon$-dense in $P_n$. Adding points in the slow regions of $\vr{P_n}{r}$ cannot create any new periodic orbits, and so we are left with precisely $z$ periodic orbits. An application of Proposition~\ref{prop:vrchtpy} now completes the proof of (ii). \qedhere

\end{proof}

\begin{remark}
If Conjecture~\ref{conj:monotonic} is true, then a version of Theorem~\ref{thm:secondary} will also be true for $n$ not necessarily a multiple of $2l+1$. This proof would, in addition to the continuity of Lemma~\ref{lem:g_r-cont-p}, also require the continuity of Lemma~\ref{lem:G}.
\end{remark}

\section{Topological Bounds on the Gromov--Hausdorff distance between $P_n$ and $S^1$}\label{sec:GH}

In this section we explore how close the lower bound that comes from persistent homology on the Gromov--Hausdorff distance between $P_n$ and $S^1$ is to being tight. Here both $P_n$ and $S^1$ are equipped with the Euclidean distance.

We start by defining some concepts. Let $(Z,d_Z)$ be a metric space. The \emph{Hausdorff distance}~\cite{burago} between any two closed subsets $X,Y \subseteq Z$ is defined as
\[ d_H^Z(X,Y) = \max(\sup_{x\in X}\inf_{y\in Y}d_Z(x,y), \sup_{y\in Y}\inf_{x\in X}d_Z(x,y)). \]

The Gromov-Hausdorff distance is a generalization of the Hausdorff distance that measures the dissimilarity between metric spaces, while accounting for possible realignment. The \emph{Gromov-Hausdorff distance}~\cite{burago} between two metric spaces $(X,d_X)$ and $(Y,d_Y)$ is defined as
\begin{align*}
d_{GH}(X,Y) = \inf_{Z,f,g} \{d_H^Z(f(X),g(Y))~|~Z \text{ a metric space, } f\colon X \to Z, \, g\colon Y \to Z \text{ isometric embeddings} \}.
\end{align*}

Note that if $X$ and $Y$ are metric spaces equipped with isometric embeddings into a common metric space $Z$, then by definition we have
\begin{equation}\label{eq:H-GH}
d_{GH}(X,Y)\le d_H^Z(X,Y).
\end{equation}

\begin{example}\label{ex:H-dist}
Let $S^1$ be the circle of unit radius in $\R^2$ about the origin, and let $P_n$ be an inscribed regular polygon with $n$ sides. We equip both sets with the Euclidean metric.

\begin{center}
\def\svgwidth{0.3\linewidth}
\begingroup%
  \makeatletter%
  \providecommand\color[2][]{%
    \errmessage{(Inkscape) Color is used for the text in Inkscape, but the package 'color.sty' is not loaded}%
    \renewcommand\color[2][]{}%
  }%
  \providecommand\transparent[1]{%
    \errmessage{(Inkscape) Transparency is used (non-zero) for the text in Inkscape, but the package 'transparent.sty' is not loaded}%
    \renewcommand\transparent[1]{}%
  }%
  \providecommand\rotatebox[2]{#2}%
  \newcommand*\fsize{\dimexpr\f@size pt\relax}%
  \newcommand*\lineheight[1]{\fontsize{\fsize}{#1\fsize}\selectfont}%
  \ifx\svgwidth\undefined%
    \setlength{\unitlength}{216.12756684bp}%
    \ifx\svgscale\undefined%
      \relax%
    \else%
      \setlength{\unitlength}{\unitlength * \real{\svgscale}}%
    \fi%
  \else%
    \setlength{\unitlength}{\svgwidth}%
  \fi%
  \global\let\svgwidth\undefined%
  \global\let\svgscale\undefined%
  \makeatother%
  \begin{picture}(1,1)%
    \lineheight{1}%
    \setlength\tabcolsep{0pt}%
    \put(0,0){\includegraphics[width=\unitlength,page=1]{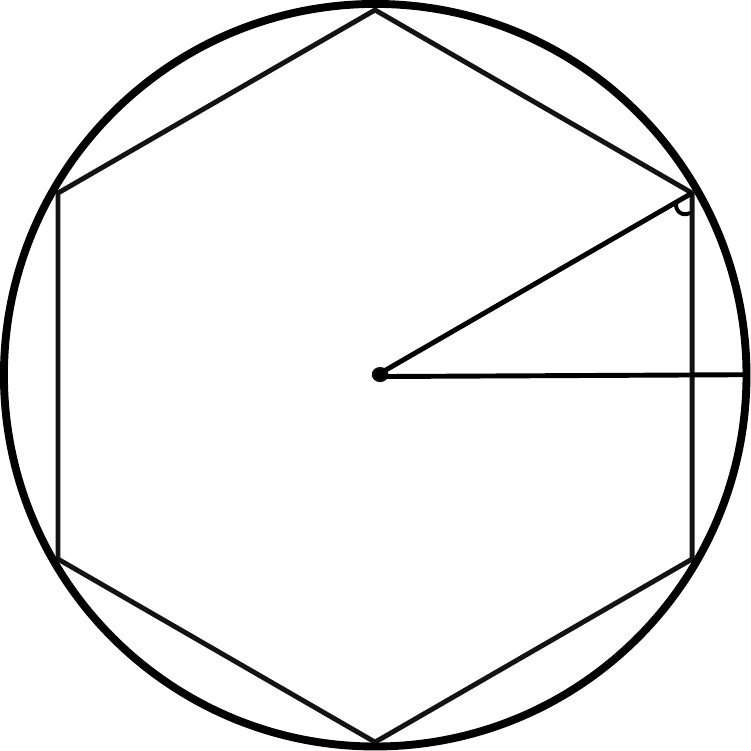}}%
    \put(0.86585878,0.63570795){\color[rgb]{0,0,0}\makebox(0,0)[lt]{\lineheight{1.25}\smash{\begin{tabular}[t]{l}$\theta$\end{tabular}}}}%
  \end{picture}%
\endgroup%

\end{center}
Each interior angle of $P_n$ is $\frac{\pi(n-2)}{n}$, and so we have $\theta = \frac{\pi(n-2)}{2n}$ in the figure. One can compute
\[d_H^{\R^2}(P_n,S^1) = 1 - \sin\bigl(\tfrac{\pi(n-2)}{2n}\bigr) = 1 - \cos(\tfrac{\pi}{n}). \]
In particular, by \eqref{eq:H-GH} we have \begin{equation}\label{ex:H-dist-Pn}
d_{GH}(P_n,S^1) \leq 1 - \cos(\tfrac{\pi}{n}). 
\end{equation}
\end{example}

We conjecture that the upper bound is tight.
\begin{conjecture}
We conjecture that $d_{GH}(P_n,S^1)=d_H^{\R^2}(P_n,S^1)=1-\cos(\tfrac{\pi}{n})$.
\end{conjecture}
To our knowledge this conjecture is not known, and therefore we search for useful lower bounds.

For $X$ a metric space, let $\PH_i^{\VR}(X)$ denote the $i$-dimensional persistent homology barcodes of $\vr{X}{r}$ over all scale parameters $r$. 
By stability of the persistent homology of Vietoris-Rips complexes \cite{ChazalDeSilvaOudot2013}, we know that for compact metric spaces $X$ and $Y$, we have
\begin{equation}\label{eq:B-GH}
d_B(\PH_i^{\VR}(X),\PH_i^{\VR}(Y))\le 2d_{GH}(X,Y).
\end{equation}

\begin{example}
For any integer $n \ge 6$ divisible by three, consider the barcodes for the 1-dimensional persistent homology of $P_n$. There is one bar in $\PH_1^{\VR}(P_n)$, with birth time $0$ and death time $s_{n,1}=\sqrt{3}\cos(\frac{\pi}{n})$ (Remark~\ref{rem:expl-sl}). Also, if $S^1$ is equipped with the Euclidean metric, then there is one bar in $\PH_1^{\VR}(S^1)$, with birth time $0$ and death time $\sqrt{3}$.
Therefore,
\[ d_B(\PH_1^{\VR}(P_n),\PH_1^{\VR}(S^1)) = \sqrt{3}(1-\cos(\tfrac{\pi}{n})). \]
We make a brief remark about the preceding calculation: the equality follows by noting that the bottleneck distance is achieved by simply matching the two nontrivial bars, and not matching with any ``trivial'' bars. We point the reader to~\cite{chazal2016structure} for details on bottleneck distance computations.

Combining the preceding observation with~\eqref{ex:H-dist-Pn} and \eqref{eq:B-GH}, we get the bounds
\[ \tfrac{\sqrt{3}}{2}(1-\cos(\tfrac{\pi}{n})) \le d_{GH}(P_n,S^1) \le 1-\cos(\tfrac{\pi}{n}). \]
This gives us a reasonably tight range on the Gromov-Hausdorff distance, for example, $d_{GH}(P_6,S^1) \in [0.116,0.134]$. Note that the upper and lower bounds both go to 0 as $n \to \infty$.
\end{example}

Now we compare these topology-driven lower bounds to those we could obtain purely via tools in metric geometry. We start with a reformulation of the Gromov-Hausdorff distance. 

Given any two sets $X$ and $Y$, a \emph{correspondence} is a subset $R\subseteq X\times Y$ such that $\pi_X(R) = X$ and $\pi_Y(R) = Y$, where the $\pi_\bullet$ maps denote the canonical projections. Let $\mathcal{R}(X,Y)$ denote the collection of all correspondences between $X$ and $Y$. The Gromov-Hausdorff distance between two metric spaces $(X,d_X), (Y,d_Y)$ can be formulated \cite{burago} as follows:
\begin{align*}
d_{GH}(X,Y) = \frac{1}{2}\inf_{R \in \mathcal{R}(X,Y)}\sup \{| d_X(x,x') - d_Y(y,y')|~|~(x,y),(x',y') \in R\}.
\end{align*}

Via the results in \cite{memoli2012some}, one obtains the following lower bounds for $d_{GH}$:
\begin{align}
d_H^{\R}(\image(d_X),\image(d_Y)) 
&\leq 
\inf_{R\in \mathcal{R}(X,Y)}\sup_{(x,y)\in R} d_H^{\R}(\image(d_X(x,\cdot), \image(d_Y(y,\cdot))
\label{eq:lb-weak}
\\
& \leq 2d_{GH}(X,Y). \label{eq:lb-strong}
\end{align}

Here $\image(d_X)$, for example, denotes the image of the map $d_X\colon X\times X\to \R$. 
We now compare the lower bounds for $d_{GH}(P_n,S^1)$ obtained via these purely metric geometry tools with those obtained via the contributions of the current paper. We temporarily adopt the notation $d_{P_n}$ and $d_{S^1}$ to denote the restrictions of the Euclidean distance to $P_n$ and $S^1$, respectively.

For any even $n\geq 6$, we know that $\diam(P_n) = 2$. Thus $\image(d_{P_n}) = [0,2] = \image(d_{S^1})$. This shows that the lower bound given by the left hand side of~\eqref{eq:lb-weak} does not discriminate between $P_n$ and $S^1$ in the even case. 

Next we test the lower bound given by~\eqref{eq:lb-strong}. Let $n \geq 6$ be even. Fix the correspondence $R$ given by matching the unique points of $P_n$ and $S^1$ that lie on a ray based at the center of $S^1$. Let $m \in P_n$ be the midpoint of an edge of $P_n$. Let $v$ denote one of the vertices of $P_n$ lying on the edge opposite $m$. Then $\norm{v - m} = \sqrt{1 + 3\cos^2(\pi/n)} = \max\image(d_{P_n}(m,\cdot))$. 

Let $p \in S^1$ be such that $(m,p) \in R$. Note that $\image(d_{S^1}(x,\cdot)) = [0,2]$. By symmetry considerations, we know that $\sup_{(x,y) \in R}d_H^{\R}(\image(d_{P_n}(x,\cdot), \image(d_{S^1}(y,\cdot))$ is achieved by $d_H^{\R}(\image(d_{P_n}(m,\cdot), \image(d_{S^1}(p,\cdot))$. Thus we have
\begin{align*}
\inf_{S\in \mathcal{R}(P_n,S^1)}\sup_{(x,y)\in S} d_H^{\R}(\image(d_{P_n}(x,\cdot), \image(d_{S^1}(y,\cdot)) 
&\leq d_H^{\R}(\image(d_{P_n}(m,\cdot), \image(d_{S^1}(p,\cdot))\\
&= d_H^{\R}([0,\sqrt{1 + 3\cos^2(\pi/n)}],[0,2]).
\end{align*}
In the case $n = 6$, we have $\sqrt{1 + 3\cos^2(\pi/n)} = \sqrt{13}/2$, and so the right hand side above evaluates to $2 - \sqrt{13}/2$. Dividing by 2, we get a lower bound of at most $0.0986$ on the Gromov-Hausdorff distance between $P_n$ and $S^1$. Thus the lower bound obtained via~\eqref{eq:lb-strong} is strictly weaker than the bound of 0.116 obtained via persistent homology considerations.

\section{Conclusion}

We end with some open questions motivated by this work.

\begin{question}
When the graph $\vr{P_n}{r}$ is not cyclic, what can be said about the homotopy type of the simplicial complex $\vrc{P_n}{r}$?
\end{question}

\begin{question}
For $l>1$ and $n$ not a multiple of $P_n$, what is an analytic formula for $s_{n,l}$ and $t_{n,l}$? This question is related to proving Conjecture~\ref{conj:monotonic} in the cases where it remains open.
\end{question}

Let $V=S^1$ be the circle of unit radius, equipped with the Euclidean metric. It is necessary that $s_{n,l}\to s_{V,l}$ and $t_{n,l}\to t_{V,l}$ as $n\to \infty$, by virtue of the fact that $P_n$ converges to $S^1$ as $n\to\infty$ (for example in the Hausdorff distance).

As explained in the introduction, we think of the following question as an analogue of Latschev's Theorem~\cite{Latschev2001} at higher scale parameters.

\begin{question}\label{ques:Latschev-higher}
Let $M$ be a Riemannian manifold, and let scale $r>0$ be arbitrary. Is it true that for $X\subseteq M$ sufficiently dense (depending on $M$ and $r$), we have a homotopy equivalence $\vrcless{X}{r}\simeq\vrcless{M}{r}$?
\end{question}

Latschev's Theorem provides a positive answer when $r$ is sufficiently small depending on $M$, and~\cite{AA-VRS1} provides a positive answer at all scales $r$ when $M=S^1$ is the circle. Related questions for the fundamental group are considered in~\cite{virk2017approximations}. Our Theorem~\ref{thm:secondary}(ii) provides a negative answer to Question~\ref{ques:Latschev-higher} if the assumption that $M$ is a Riemmanian manifold is removed.

\section{Acknowledgements}

We thank Bowen Li and Zhi Li from Colorado State University for their proof of Lemma~\ref{lem:technical-inequality}. This material is based upon work supported by the National Science Foundation under Grant No.\ DMS-1439786 while the authors were in residence at the Institute for Computational and Experimental Research in Mathematics in Providence, RI, during the Summer\MVAt ICERM 2017 program. While in residence at Summer\MVAt ICERM 2017, Bonginkosi Sibanda was also supported by The Karen T.\ Romer Undergraduate Teaching and Research Awards.

\bibliographystyle{plain}
\bibliography{VRofPolygon}

\appendix

\section{Proofs of technical lemmas}

\begin{lemma}\label{lem:technical-inequality}
For all integers $n\ge 4$ we have $\frac{n-1}{n+1}<\cos(\frac{\pi}{n})$.
\end{lemma}

\begin{proof}
Let $f\colon (0,\infty)\to\R$ be defined by $f(x)=\cos(\frac{\pi}{x})-\frac{x-1}{x+1}$. We compute
\[ f'(x)=\frac{\pi\sin(\frac{\pi}{x})(x+1)^2-2x^2}{x^2(x+1)^2}. \]
Note the denominator of $f'(x)$ is positive for all $x\in(0,\infty)$. To get a bound on the numerator, we observe that $\sin(\frac{\pi}{x})\le\frac{\pi}{x}$ gives
\[ \pi\sin(\tfrac{\pi}{x})(x+1)^2-2x^2 \le \tfrac{\pi^2}{x}(x+1)^2-2x^2:=g(x). \]
Since $g(7)<0$ and $g'(x)=\pi^2(1-\frac{1}{x^2})-4x\le \pi^2-4x<0$ for all $x\ge 7$, we get that $g(x)<0$ for all $x\ge 7$ and hence $f'(x)<0$ for all $x\ge 7$. Since $\lim_{x\to\infty}f(x)=0$, this implies that $f(x)>0$ (and hence $\frac{x-1}{x+1}<\cos(\frac{\pi}{x})$) for all $x\ge 7$. We now confirm the remaining cases $n=4,5,6$ individually in order to obtain $\frac{n-1}{n+1}<\cos(\frac{\pi}{n})$ for all integers $n\ge 4$.
\end{proof}

\begin{lemma}\label{lem:sin-cos-inequality}
The inequality $\sin(2\pi/x) \le \cos(\pi/x)$ holds for all $x \ge 6$, with equality if and only if $x = 6$.
\end{lemma}

\begin{proof}
Consider the function $f: [6,\infty)\to\mathbb{R}$ defined by $f(x) = \cos(\pi/x) - \sin(2\pi/x)$. Then simply note that we have $f(6) = 0$ and $f'(x) = \frac{\pi}{x^2}\sin(\pi/x) + \frac{2\pi}{x^2}\cos(2\pi/x) > 0$ for $x \ge 6$.
\end{proof}

\section{Proof of Lemma~\ref{lem:monotonic-1}}\label{sec:monotonic-1}

We want to prove Lemma~\ref{lem:monotonic-1}, i.e.\ that the monotonicity result for $s_{2l+1}$ presented in Conjecture~\ref{conj:monotonic} is true for $l = 1$. Throughout this section, we fix $l = 1$, and therefore $n \geq 6$. 

Given points $A,B\in \R^2$, we write $AB$ to denote both the edge and the corresponding edge length, where the meaning will be clear from context. We will also write terms such as $\sin(ABC)$ or $\cos(ABC)$ to mean the sines and cosines of the angle $\angle ABC$.

Let $x \in P_n$. Then $S_{2l+1}(x)$ is an inscribed equilateral triangle. We adopt some changes in notation for convenience. We define $PQR:=S_{2l+1}(x)$, i.e., we denote the vertices of $S_{2l+1}(x)$ by $P$, $Q$, and $R$.

\begin{figure}
\def\svgwidth{0.4\linewidth}
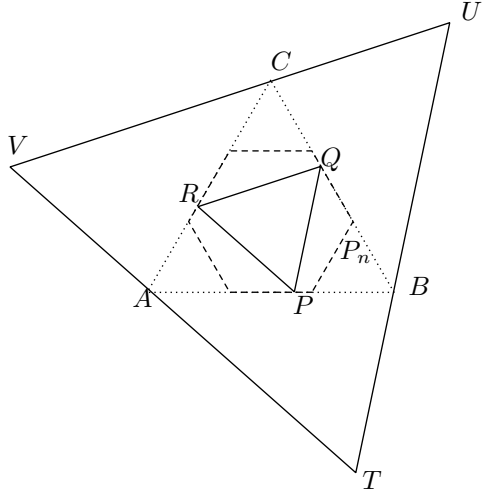
\caption{For any triangle $\triangle PQR$ inscribed in a polygon, the edges containing $P,Q,R$ can be extended to a triangle $\triangle ABC$ that contains $\triangle PQR$ as an inscribed triangle. Let $TUV$ be the equilateral triangle circumscribed about $ABC$ that is parallel to $PQR$.}
\label{fig:polygon-triangle-inscribed}
\end{figure}

The three edges of $P_n$ on which the three vertices $P$, $Q$, and $R$ lie can be extended to form a triangle $\triangle ABC$, as shown in Figure~\ref{fig:polygon-triangle-inscribed}. 
Let $TUV$ be the equilateral triangle circumscribed about $ABC$ that is parallel to $PQR$. We will show that the following relationship holds between the side lengths of the inscribed and circumscribed triangles: $RQ \cdot VU = k_{ABC}$, where $k_{ABC}$ is a constant depending only on triangle $ABC$. Moreover, the side length $VU$ is given by a function of cosine, which is monotonic between its extrema. This will give us monotonicity information about $RQ$ so long as $PQR$ is inscribed in $ABC$. 

We found these ideas in the writings of Philippe Chevanne~\cite{chevanne}. It is possible that these results are well-known; nevertheless, we repeat the constructions below.

\begin{claim}\label{claim:B1}
Using the notation as in Figure~\ref{fig:polygon-triangle-inscribed}, we have
\[\frac{AB}{RQ} = \frac{\sin(ARP)}{\sin(RAP)} + \frac{\sin(BQP)}{\sin(QBP)}.\]
\end{claim}

\begin{proof} Via the sine rule for triangles $ARP$ and $BQP$, we have
\[AB = AP + PB = PR\frac{\sin(ARP)}{\sin(RAP)} + PQ\frac{\sin(BQP)}{\sin(QBP)} = RQ\left(\frac{\sin(ARP)}{\sin(RAP)} + \frac{\sin(BQP)}{\sin(QBP)} \right), \]
where the last step follows since $RQ = PQ = PR$.
\end{proof}

\begin{claim} 
\label{cl:triangles-max-min}
\[RQ\cdot VU = AB^2\frac{\sin(ABC)}{\sin(ACB)}\frac{\sin(BAC)}{\sin(\pi/3)}.\]
\end{claim}

\begin{proof}
We use the sine rule along with the fact that $PQ$ is parallel to $TU$ and $PR$ is parallel to $TV$. Note
\begin{align*}
VU &= VC + CU\\
&= AC\frac{\sin(VAC)}{\sin(AVC)} + BC\frac{\sin(CBU)}{\sin(BUC)}\\
&= AC\frac{\sin(ARP)}{\sin(AVC)} + BC\frac{\sin(BQP)}{\sin(BUC)} \tag{parallel sides} \\
&= AB\frac{\sin(ABC)}{\sin(ACB)} \frac{\sin(ARP)}{\sin(AVC)} + AB\frac{\sin(BAC)}{\sin(ACB)}
\frac{\sin(BQP)}{\sin(BUC)} \tag{sine rule for $\triangle ABC$} \\
&= AB\frac{\sin(ABC)}{\sin(ACB)} \frac{\sin(ARP)}{\sin(\pi/3)} + AB\frac{\sin(BAC)}{\sin(ACB)}
\frac{\sin(BQP)}{\sin(\pi/3)}  \tag{$\triangle TUV$ is equilateral} \\
&= AB\frac{\sin(ABC)}{\sin(ACB)} \frac{\sin(BAC)}{\sin(\pi/3)} 
\left( \frac{\sin(ARP)}{\sin(BAC)} + \frac{\sin(BQP)}{\sin(ABC)} \right) \\
&= AB\frac{\sin(ABC)}{\sin(ACB)} \frac{\sin(BAC)}{\sin(\pi/3)} 
\left(\frac{\sin(ARP)}{\sin(RAP)} + \frac{\sin(BQP)}{\sin(QBP)}\right) \\
&= AB\frac{\sin(ABC)}{\sin(ACB)} \frac{\sin(BAC)}{\sin(\pi/3)} 
\left( \frac{AB}{RQ}\right).\tag{by Claim~\ref{claim:B1}} 
\end{align*}
\end{proof}

It follows that the sizes of the circumscribed and inscribed triangles are inversely proportional to each other. Next we show that the side length $UV$ varies as a function of cosine as the position of the inscribed triangle $\triangle PQR$ varies along $\triangle ABC$.

\begin{figure}
\includegraphics[width = 0.7\textwidth]{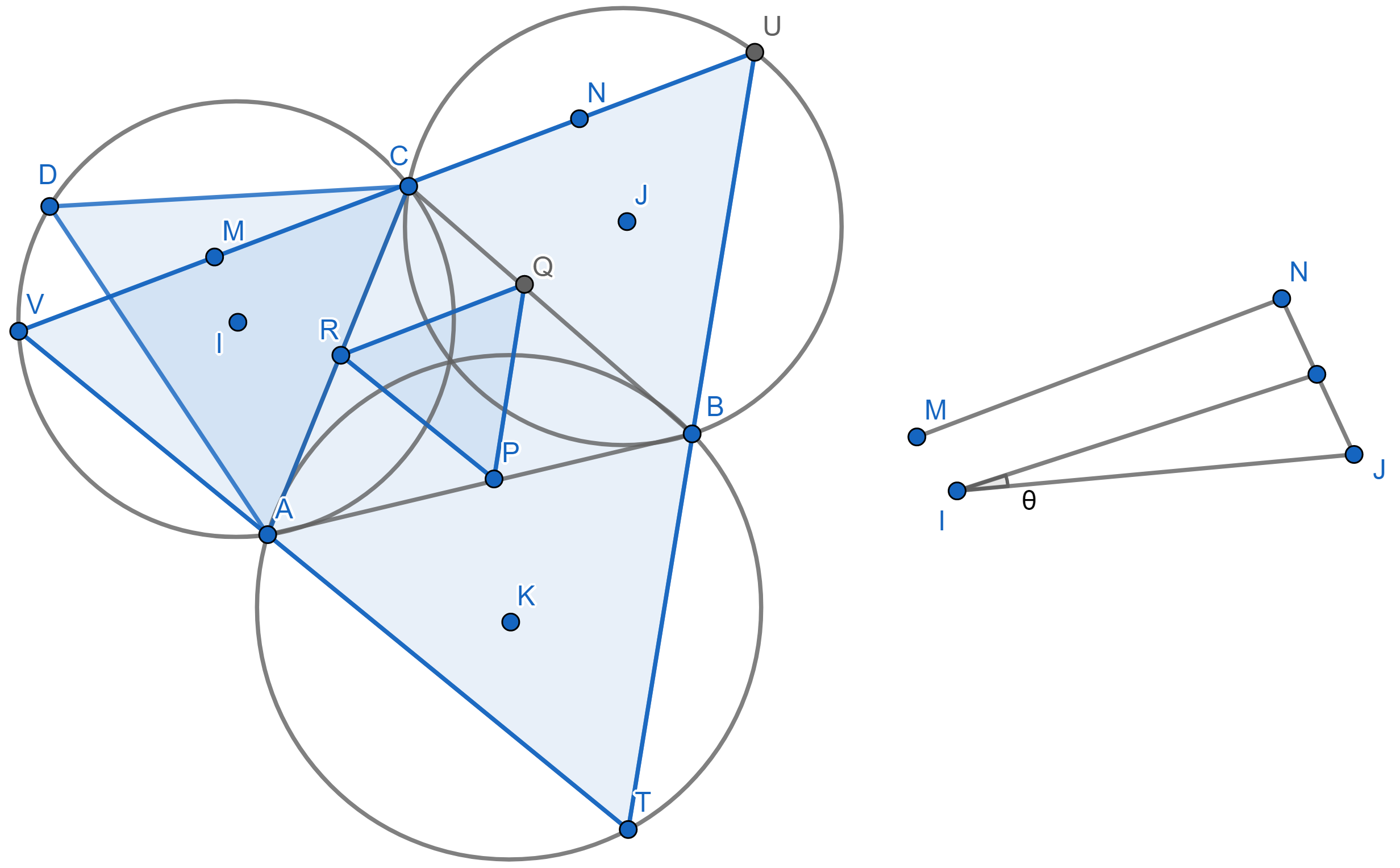}
\caption{(Left) $I$, $J$, and $K$ form the centers of the outer Napoleon triangle of $\triangle ABC$. The vertices $M$ and $N$ are midpoints of $VC$ and $CU$, respectively. (Right) angle between $IJ$ and $VU$.}
\label{fig:triangle-inscribed-napoleon}
\end{figure}

We construct the \emph{outer Napoleon triangle} of $\triangle ABC$. 
Consider an equilateral triangle $\triangle ACD$ such that $D$ lies outside $\triangle ABC$. Let $I$ denote the centroid of this triangle. Separately consider the circumscribed circle of the triangle $\triangle AVC$, as shown in Figure~\ref{fig:triangle-inscribed-napoleon}. Now $\angle AVC = \pi/3 = \angle ADC$. Thus by the inscribed angle theorem, $V$ and $D$ are both points on the arc of a circle that also contains $A$ and $C$. Furthermore, the center of this circle is precisely the centroid $I$ of the equilateral triangle $\triangle ACD$.

Similarly, one obtains the points $J,K$ depending only on $BC$ and $AB$, respectively. The equilateral triangle $\triangle IJK$ is the \emph{outer Napoleon triangle} of $\triangle ABC$. 

Next let $\theta$ denote the angle between the edge $IJ$ and the edge $VU$, as shown in the right of Figure~\ref{fig:triangle-inscribed-napoleon}.

\begin{claim} $VU = 2(IJ)\cos(\theta)$.
\end{claim}

\begin{proof}
Let $M,N$ denote the midpoints of $VC$ and $CU$. Then $VU = 2(MN) = 2(IJ)\cos(\theta)$.
\end{proof}

Now $VU$ is maximized when $\theta = 0$, i.e.\ when $VU$ is parallel to $IJ$. By Claim~\ref{cl:triangles-max-min}, we have the following.

\begin{corollary}
\label{cor:max-min-napoleon}
$RQ$ is minimized when $VU$ is maximized, i.e.\ when $\theta = 0$, or equivalently, when $\triangle TUV$ and $\triangle PQR$ are parallel to the Napoleon triangle $\triangle IJK$ of $\triangle ABC$. Furthermore, $RQ$ is monotonically increasing away from its minimizer.
\end{corollary}

The next claim is the content of Conjecture~\ref{conj:monotonic}.

\begin{claim}
Every midpoint crossing is a global minimizer of $s_{2l+1}(x)$, every vertex crossing is a global maximizer of $s_{2l+1}(x)$, the midpoint and vertex crossings are interleaved around $P_n$, and $s_{2l+1}$ is strictly monotonic between adjacent midpoint and vertex crossings.
\end{claim}

\begin{proof}
Let $x_0$ denote a midpoint crossing, and let $P,Q,R$ denote the vertices of $S_{2l+1}(x_0)$. By Lemma~\ref{lem:crossing-extrema}, $x_0$ is a local extrema of $s_{2l+1}(x)$. We claim that $s_{2l+1}$ has a local minimum at $x_0$. 
Towards a contradiction, suppose $s_{2l+1}$ has a local maximum at $x_0$. We know by Lemma~\ref{lem:vertex-midpt-impossible} that $x_0$ is not a vertex crossing, thus there exists a small $P_n$-neighborhood of $x_0$ along which $x$ can vary so that the vertices of $S_{2l+1}(x)$ and $S_{2l+1}(x_0)$ belong to the same edges of $P_n$. 
In particular, these vertices belong to the same triangle $\triangle ABC$ obtained by extending the edges of $P_n$. Consider the Napoleon triangle $\triangle IJK$ of $\triangle ABC$. 
By the previous work and the assumption that $s_{2l+1}$ does not have a local minimum at $x_0$, we know that $\triangle IJK$ is not parallel to $\triangle PQR$. In a small neighborhood of $x_0$, as $x$ traverses $x_0$ in a counterclockwise direction, $S_{2l+1}(x)$ becomes either more parallel or less parallel to $\triangle IJK$. Thus our previous work shows that $s_{2l+1}(x)$ varies strictly monotonically. But this contradicts the assumption that $s_{2l+1}$ has a local maximum at $x_0$. It follows that $s_{2l+1}$ has a local minimum at $x_0$, and moreover that $\triangle PQR$ is parallel to $\triangle IJK$. 

As we move the point $x\in P_n$ counterclockwise starting at midpoint crossing $x_0$, we know by Corollary~\ref{cor:max-min-napoleon} that $s_{2l+1}(x)$ is strictly increasing until $x$ reaches a vertex crossing $x_1$. At this point there are multiple triangles (each with its own Napoleon triangle) inside which $S_{2l+1}(x_1)$ can be inscribed. By symmetry and Lemma~\ref{lem:crossing-extrema}, we know that $s_{2l+1}(x)$ decreases again after $x$ passes $x_1$. In particular, $s_{2l+1}$ has a local maximum at the vertex crossing $x_1$. 

We claim that the next crossing $x_2$ hit by $x$ during this counterclockwise traversal is a midpoint crossing. Towards a contradiction, suppose $x_2$ is another vertex crossing. Since $x_1$ was a local maximizer and $s_{2l+1}$ is strictly monotonic, we know that $s_{2l+1}(x)$ strictly decreases as $x$ approaches $x_2$. But the value of $s_{2l+1}$ must be equal at all vertex crossings by Lemma~\ref{lem:crossing-extrema}, and thus $x_2$ cannot be a vertex crossing. It follows that $x_2$ is a midpoint crossing. 

By Lemma~\ref{lem:crossing-extrema}, we know that $s_{2l+1}$ has the same value at all vertex crossings and all midpoint crossings, respectively. Thus the local extrema are all global extrema. We have shown that vertex and midpoint crossings are interleaved around $P_n$, vertex crossings are global maximizers of $s_{2l+1}$, midpoint crossings are global minimizers of $s_{2l+1}$, and that $s_{2l+1}$ is strictly monotonic between midpoint and vertex crossings that are adjacent in $P_n$. This concludes the proof. \end{proof}

\end{document}